\documentclass[a4paper,12pt]{article}
\usepackage[utf8]{inputenc}

\usepackage[T1]{fontenc}
\usepackage{verbatim}
\usepackage[margin=3cm]{geometry}
\usepackage[algo2e,ruled,vlined]{algorithm2e}
\usepackage{algorithm}
\usepackage{here}
\usepackage{setspace}
\usepackage[normalem]{ulem}
\usepackage{algpseudocode}
\usepackage{dsfont}
\usepackage{url}
\usepackage[inline]{enumitem}
\usepackage{microtype}
\usepackage[affil-sl]{authblk}
\usepackage{graphicx}
\usepackage{nicematrix}
\usepackage{amsmath}
\usepackage{amssymb}
\usepackage{amsthm}
\usepackage{mathtools}
\usepackage{nicefrac}
\usepackage{amsfonts,latexsym,graphics}
\usepackage{color}
\usepackage[hidelinks]{hyperref}
\usepackage{multirow}
\usepackage[english]{babel}
\usepackage{epsfig}
\usepackage{bm}
\usepackage{makeidx}
\usepackage{subcaption}
\usepackage{adjustbox}
\usepackage{booktabs}

\usepackage{tikz}
\usetikzlibrary{tikzmark}
\usetikzlibrary{positioning}
\usetikzlibrary{graphs}
\usetikzlibrary{graphs.standard}
\usetikzlibrary{decorations.pathmorphing}
\usetikzlibrary{decorations.text}
\pgfdeclarelayer{bg}    
\pgfsetlayers{bg,main}  

\usepackage{xspace}         

\newcommand{\norm}[2]{\| #1 \|_{#2}}
\newcommand{\define}{\coloneqq}

\newcommand{\R}{\mathds{R}}

\def\A{{\mbox {\boldmath $A$}}}

\def\matrix0{{\mbox {\boldmath $O$}}}

\def\vece{{\mbox{\boldmath $e$}}}
\def\vecu{{\mbox{\boldmath $u$}}}
\def\vecv{{\mbox{\boldmath $v$}}}

\def\vecx{{\mbox{\boldmath $x$}}}
\def\vecy{{\mbox{\boldmath $y$}}}

\def\j{{\mbox{\boldmath $1$}}}
\def\vec0{\mbox{\bf 0}}

\def\vecnu{{\mbox{\boldmath $\nu$}}}
\def\vecchi{{\mbox{\boldmath $\chi$}}}

\theoremstyle{definition}
\newtheorem{definition}{Definition}
\newtheorem{example}[definition]{Example}

\newtheorem{remark}[definition]{Remark}

\theoremstyle{plain}

\newtheorem{lemma}[definition]{Lemma}
\newtheorem{theorem}[definition]{Theorem}

\newtheorem{corollary}[definition]{Corollary}
\newcommand*{\claimproofname}{Proof}

\newcommand\tran{\mkern-2mu\raise1.25ex\hbox{$\scriptscriptstyle\top$}\mkern-3.5mu}

\def\Pa{{\mathcal{P}}}

\def\supp{\mathop{\rm supp }\nolimits}

\newcommand{\vol}[1]{\operatorname{vol} \left(#1\right)}
\newcommand{\rt}[1]{\operatorname{rt} \left(#1\right)}

\DeclarePairedDelimiter\abs{\lvert}{\rvert}
\DeclarePairedDelimiter\inn{\langle}{\rangle}

	{\end{list}}

\SetArgSty{}
\SetKwFor{For}{for}{do}{}
\SetKwInOut{Input}{Input}
\SetKwInOut{Output}{Output}

\makeatletter
\let\c@figure\c@table
\let\ftype@figure\ftype@table
\let\ext@figure\ext@table

\makeatother 

\title{A unified framework for the Expander Mixing Lemma for irregular graphs and its applications}

\author{Aida Abiad
\thanks{\texttt{a.abiad.monge@tue.nl}, Department of Mathematics and Computer Science, Eindhoven University of Technology, The Netherlands}
\qquad \qquad Sjanne Zeijlemaker\thanks{\texttt{s.zeijlemaker@tue.nl},  Department of Mathematics and Computer Science, Eindhoven University of Technology, The Netherlands}}

\date{}

\begin{document}
\maketitle

\begin{abstract}
A unified framework for the Expander Mixing Lemma for irregular graphs using adjacency eigenvalues is presented, as well as two new versions of it. While the existing Expander Mixing Lemmas for irregular graphs make use of the notion of volume (the sum of degrees within a vertex set), we instead propose to use the Perron eigenvector entries as vertex weights, which is a way to regularize the graph. This provides a new application of weight partitions of graphs. The new Expander Mixing Lemma versions are then applied to obtain several eigenvalue bounds for NP-hard parameters such as the zero forcing number, the vertex integrity and the routing number of a graph. 
\end{abstract} 

\section{Introduction}

 The Expander Mixing Lemma (EML) has served as an important tool relating the spectrum of a graph to the edge distribution and notions of pseudorandomness. Intuitively, it says that a regular graph behaves like a random graph whenever the second largest eigenvalue in absolute value is small. The Expander Mixing Lemma is often attributed to Alon and Chung~\cite{AC1988}. However, similar results were shown previously by Bussemaker, Cvetkovi\'c and Seidel~\cite{BCS1978} and Haemers~\cite{haemersThesis}. In his PhD thesis, Haemers also proved an Expander Mixing Lemma for semiregular graphs using eigenvalue interlacing~\cite{haemersThesis}. A stronger version of the original lemma by Alon and Chung was later proposed by Krivelevich and Sudakov~\cite{KS2006}.

The popularity of the  Expander Mixing Lemma has also inspired several versions for irregular graphs based on the adjacency eigenvalues. Krivelevich and Sudakov~\cite{KS2006} noted that their Expander Mixing Lemma for regular graphs can be extended to general graphs. Chung~\cite{chung2004} also proposed a similar statement that closely resembles the original proof by Krivelevich and Sudakov. For the Laplacian matrix, an Expander Mixing Lemma was proved by Chung~\cite{chung2004}. 
Such Expander Mixing Lemmas for irregular graphs have been shown to be useful tools to study other graph properties. Instances of it are the recent work of Byrne and Tait~\cite{BT2023}, who used the version by Krivelevich and Sudakov to improve upper bounds on even-cycle creating Hamilton paths, or the application of the irregular Expander Mixing Lemma by Chung to bound the toughness of a graph~\cite{chen2022,huang2022toughness}, upper bound the routing number~\cite{horn2020routing} and derive spectral conditions for Hamiltonicity~\cite{fan2012spectral}.  

Given the many variants of the Expander Mixing Lemma for irregular graphs, their similarity in form and their ad hoc applications, it makes sense to aim for a generalized framework that unifies these results. More general Expander Mixing Lemmas for (not necessarily square) matrices have been proposed by Butler~\cite{Butler2006} and Nikiforov~\cite{nikiforov2009}. In this article, we propose a stronger generalized Expander Mixing Lemma, which encompasses the known variants from the literature. In addition, we derive two new Expander Mixing Lemmas for irregular graphs from our main result based on a new kind of centrality measure, as well as their converse statement. While the existing Expander Mixing Lemmas for irregular graphs by Krivelevich and Sudakov, and by Chung, make use of the notion of volume (the sum of degrees within a vertex set), in this article we instead propose to use the entries of the Perron eigenvector of the adjacency matrix to assign weights to the vertices. Doing so, we ``regularize'' the graph, in the sense that the weight-degree of each vertex becomes a constant. This lies in the core of the theory of weight partitions of a graph, which has been shown to be a powerful method to extend several results of spectral nature to general graphs, see e.g.~\cite{A2019,F1999,FG1999,haemers1979eigenvalue,LW2012}. We also propose a converse to both statements, i.e., we show that if the number of edges between any two vertex sets is close to the expected distribution in a random graph, the second largest eigenvalue in absolute value is bounded. Such results have been proposed for several other types of Expander Mixing Lemmas, see~\cite{BL2005,bollobas2004,kenter2016,Butler2006}.

The structure of this paper is as follows. In Section \ref{sec:existing}, we start with an overview of some of the known Expander Mixing Lemmas for irregular graphs. We then propose a general Expander Mixing Lemma for matrices in Section~\ref{sec:expandermixingnew}, which unifies these results. As a corollary, we find several interesting new variants of the Expander Mixing Lemma for irregular graphs based on the Perron eigenvector. In Section~\ref{sec:applications}, these new versions are used to find new spectral bounds on several NP-hard graph parameters, such as the zero forcing number, the vertex integrity, the routing number and the $k$-independence number. Finally, in Section~\ref{sec:converse}, we obtain a converse for the new Expander Mixing Lemma with Perron weights, generalizing the classic result of Bilu and Linial~\cite{BL2005}.

	\section{Preliminaries}\label{sec:prelim}
		
		Throughout this article, we denote by~$I$ an identity matrix,~$J$ an all-ones matrix and by~$\j$ an all-ones vector whose dimensions will be clear from the context. The Euclidean norm of a vector is denoted by~$\norm{\cdot}{}$ and we denote~$[n]\coloneqq\{1,2,\dots,n\}$.
        
        For a simple undirected connected graph~$G = (V, E)$, let~$n = |V|$
        be the number of vertices and~$A = A(G)$ its adjacency matrix.
        Throughout this article, we always assume~$G$ to be simple, undirected and connected. The set~$G(u)$ denotes the neighborhood of a vertex~$u \in V$, i.e., the set of vertices adjacent to $u$, and we write~$u \sim v$ if~$u,v \in V$ are adjacent. The closed neighborhood of~$u$ is denoted by~$G[u] \coloneqq G(u)\cup \{u\}$. The degree of a vertex~$u$ is~$d_u \coloneqq |G(u)|$. We denote the smallest and largest degree of a graph by~$\delta$ and~$\Delta$ respectively. For a subset~$S\subseteq V$, let~$\vol{S} = \sum_{v\in S}d_v$ and~$\overline{S} = V\setminus S$. The set of edges between two sets~$S,T\subseteq V$ is denoted by~$E(S,T)$ and we define~$e(S,T) = \abs{E(S,T)}$, where edges that lie in~$S\cap T$ are counted twice. We denote the characteristic vector of a vertex set~$S$ by~$\vecchi_S$.

        The adjacency eigenvalues of~$A$ are denoted by~$\lambda_1\geq \cdots \geq \lambda_n$ and we denote~$\lambda=\max\{\abs{\lambda_2},\abs{\lambda_n}\}$. Let~$D$ be the diagonal matrix with entries~$d_{vv} = d_v$. The \textit{Laplacian eigenvalues} of~$G$, $0 = \mu_0 \leq \dots \leq \mu_{n-1}$, are the eigenvalues of the matrix~$L\coloneqq D-A$ and the \textit{normalized adjacency eigenvalues}, $1=\sigma_1\ge \dots \ge \sigma_n \ge -1$, are the eigenvalues of the matrix~$D^{-1/2}AD^{-1/2}$.
     
        Since we assume graphs to be connected (so~$A$ is irreducible), the Perron-Frobenius
        Theorem assures that~$\lambda_{1}$ is simple, positive, and has a positive eigenvector. Moreover, it is the only eigenvalue that admits such an eigenvector. We denote this eigenvector by~$\vecnu=(\nu_{1},\ldots,\nu_{n})^{\top}$ and refer to it as the \emph{Perron eigenvector}. It is assumed to be normalized such that~$\norm{\vecnu}{} = 1$. If~$G$ is regular, we have~$\vecnu=\frac{1}{\sqrt{n}}\j$.
    Let~$\vece_i$ denote the $i^{\text{th}}$ standard basis vector of appropriate length. The map~$\rho\colon \mathfrak{P}(V) \to \R^n$, where~$\mathfrak{P}(V)$ denotes the power set of~$V$, is defined as $\rho(U) \define \sum_{u\in U} \nu_{u} \vece_{u}$ for any $U\neq \emptyset$. By convention, $\rho(\emptyset) = 0$ and~$\rho(u)\define \rho(\{u\})$.
 
    In a~$d$-regular graph, the rows of its adjacency matrix sum up to its largest eigenvalue,~$\lambda_1=d$. If~$G$ is not regular, we can regularize it by assigning the weight~$\rho(u)$ to each vertex~$u\in V$. This ensures that the \emph{weight-degree} $\delta_u^{\ast}$ of each vertex $u\in V$ equals the largest eigenvalue, i.e.,
        \[
          \delta^{*}_{u}
          \define
          \frac{1}{\nu_{u}}\sum_{v\in G (u)}\nu_{v}
          =
          \lambda_{1}.
        \]	 
    
	If $\mathcal{P}$ is a partition of the vertex set $V=V_{1}\cup \cdots \cup V_{m}$, the \emph{weight-intersection number}
        of $u\in V_{i}$, $i \in [m]$, is
        \begin{align*}
          b^{*}_{ij}(u)
          &\define
          \frac{1}{\nu_{u}}\sum_{v \in G(u)\cap V_{j}}\nu_{v},
          && i,j \in [m].
        \end{align*}
    The partition~$\mathcal{P}$ is \emph{weight-equitable} (or \emph{weight-regular)} if~$b^*_{ij}(u) = b^*_{ij}(v)$ for all~$i,j \in [m]$ and~$u,v \in V_i$.
    In other words, the weight-intersection numbers are independent of the choice of~$u \in V_i$.
    In this case, we write~$b^*_{ij}$ instead of~$b^{*}_{ij}(u)$, $u \in V_i$. The \emph{weight-quotient matrix} of $A$, with respect to $\mathcal{P}$, is the symmetric matrix with entries
	\begin{align*}
		\tilde{b}^{*}_{ij}=\sum_{uv \in
			E(V_{i},V_{j})}\nu_{u}\nu_{v}=\tilde{b}^{*}_{ji}
	\end{align*}
	and the \emph{normalized weight-quotient matrix} of $A$, $\bar{B}^{*}=(\bar{b}^{*}_{ij})_{i,j \in [m]}$, is defined as
    \begin{equation*}
    \bar{b}^{*}_{ij}=\frac{\tilde{b}^{*}_{ij}}{\|\rho (V_{i})\|\|\rho (V_{j})\|}.
    \end{equation*}
    For more background on weight-equitable partitions, see~\cite{A2019,F1999}.

\section{Existing Expander Mixing Lemmas for irregular graphs}\label{sec:existing}

Below, we survey some existing versions of the Expander Mixing Lemma for irregular graphs. We focus in particular on those that use the adjacency spectrum.

The Expander Mixing Lemma for regular graphs is usually attributed to Alon and Chung~\cite{AC1988}. 

\begin{theorem}[Alon, Chung \cite{AC1988}] \label{th:emlOriginal}
Let $G$ be a $d$-regular
$n$-vertex graph. Let $\lambda_1\geq \cdots \geq \lambda_n$ be the eigenvalues of the adjacency matrix of $G$ and define $\lambda = \max\{\abs{\lambda_2},\abs{\lambda_n}\}$. Then, for every $S, T \subseteq V(G)$, we have
\[\abs*{e(S,T) - \frac{d\abs{S}\abs{T}}{n}}\le \lambda \sqrt{ \abs{S}\abs{T}\left(1-\frac{\abs{S}}{n}\right)\left(1-\frac{\abs{T}}{n}\right)}.\]
 
\end{theorem}

\noindent It should be noted, however, that similar statements were shown previously by Bussemaker, Cvetkovi\'c and Seidel~\cite[Theorem 4.1]{BCS1978} and Haemers~\cite[Theorem 2.1.4(i)]{haemersThesis}. In addition, Haemers proved an Expander Mixing Lemma in terms of combinatorial designs using interlacing techniques (see~\cite[Theorem 3.1.1]{haemersThesis} and~\cite[Theorem 5.1]{H1995}). For more details on the equivalence of these results, we refer the reader to~\cite{bishnoiblog}.

Inspired by the work of the aforementioned authors, several variants of the Expander Mixing Lemma were derived that also hold for irregular graphs. We will summarize some of these results below.

In his PhD thesis, Haemers proved the following Expander Mixing Lemma for irregular graphs using eigenvalue interlacing. 

\begin{theorem}[Haemers {\cite[Theorem 2.1.2]{haemersThesis}}]
 Let~$G=(V,E)$ be a graph with adjacency  eigenvalues~$\lambda_1\ge \dots \ge \lambda_n$ and~$\lambda = \max\{\abs{\lambda_2},\abs{\lambda_n}\}$. Let $\vol{S} = \sum_{v\in S}d_v$ for $S\subseteq V$. Then for any~$S\subset V$, 
\[\abs*{e(S,\overline{S}) - \frac{\vol{S}\vol{\overline{S}}}{\vol{V}}}\le \lambda \frac{\lambda_1|S||\overline{S}|}{\vol{V}}.\]
\label{th:irregHaemers}
\end{theorem}

\noindent For regular graphs, Theorem~\ref{th:irregHaemers} was used to prove the well-known Hoffman bound on the independence number, see~\cite[Theorem 2.1.4 (ii)]{haemersThesis}.

 A proof of the Expander Mixing Lemma by Chung was also given by Krivelevich and Sudakov in~\cite{KS2006}. The following Expander Mixing Lemma for irregular graphs can be derived analogously to their original proof. This result was also proved explicitly by Chung~\cite{chung2004}.

\begin{theorem}[Chung {\cite[Theorem 3]{chung2004}}]
Let~$G=(V,E)$ be a graph and~$S$,~$T$ subsets of its vertices. Let~$\sigma_i$ be the~$i^{\text{th}}$ largest normalized adjacency eigenvalue of~$G$ and~$\sigma = \max\{\abs{\sigma_2},\abs{\sigma_n}\}$. Define $\vol{S} = \sum_{V\in S}d_v$ for any $S\subseteq V$. Then 
\[\abs*{e(S,T) - \frac{\vol{S}\vol{T}}{\vol{V}}} \le \sigma\sqrt{\vol{S}\vol{T}\left(1-\frac{\vol{S}}{\vol{V}}\right)\left(1-\frac{\vol{T}}{\vol{V}}\right)}.\]
\label{th:strongVolEML}
\end{theorem}

An analogous result to Theorem~\ref{th:strongVolEML} was proved for matrix-weighted graphs in~\cite{hansen2021}.

\begin{remark}
    Note that, in general, Theorem~\ref{th:irregHaemers} and Theorem~\ref{th:strongVolEML} are incomparable, although computational results for small graphs suggest that Theorem~\ref{th:strongVolEML} is usually stronger. For example, if~$G$ is the complete graph~$K_5$ with one edge~$\{u,v\}$ removed, the right-hand side of Theorem~\ref{th:strongVolEML} is smaller for~$S=\{u\}$, but larger if~$S = \{w\}$ for any $w\in V\setminus\{u,w\}$.
\end{remark}

The authors in~\cite{KS2006} remark that their result also extends to irregular graphs and give an outline of the steps required. The proof details were worked out recently by Byrne and Tait~\cite{BT2023}.

\begin{lemma}[Krivelevich, Sudakov {\cite{KS2006}}, Byrne, Tait~\cite{BT2023}] 
    Let~$G=(V,E)$ be a graph with minimum and maximum degrees~$\delta$ and~$\Delta$, average degree~$\overline{d}$ and adjacency eigenvalues~$\lambda_1\ge \dots \ge \lambda_n$. Let~$\lambda = \max\{\abs{\lambda_2},\abs{\lambda_n}\}$ and~$R>0$ such that~$\Delta-\delta\le R\lambda$. If~$\lambda \ll \overline{d}$, then for any~$S,T\subseteq V$,
    \[\abs*{e(S,T) - \frac{\overline{d}\abs{S}\abs{T}}{n}} \le \overline{\lambda}\sqrt{\abs{S}\abs{T}},\]
    where $\overline{\lambda} = (10R + 1)\lambda$.
    \label{lem:taitEML}
\end{lemma}

A more general approach to Expander Mixing Lemmas was taken by Butler \cite{Butler2006}, who proved a version for nonnegative matrices, as well as a converse (see Section~\ref{sec:converse}). This result was generalized to all matrices by Nikiforov~\cite{nikiforov2009}.

\begin{theorem}[Butler {\cite[Theorem 1]{Butler2006}}]
   Let~$B \in \mathbb{R}^{m\times n}$ be a matrix with nonnegative entries and no zero rows/columns. Also, let~$R\in \mathbb{R}^{m\times m}$ and~$C \in \mathbb{R}^{n\times n}$ be the unique diagonal matrices such that $B\j = R\j$ and~$\j^\top B = \j^\top C$. Then for all~$S \subseteq [m]$ and~$T \subseteq [n]$,
   \[\abs*{\langle \vecchi_S,B\vecchi_T \rangle - \frac{\langle \vecchi_S,B\j \rangle\langle\j,B\vecchi_T \rangle}{\langle \j,B\j \rangle}} \le \sigma_2(R^{-1/2}BC^{-1/2})\sqrt{\langle \vecchi_S,B\j \rangle\langle\j,B\vecchi_T \rangle},\]
   where~$\sigma_2(R^{-1/2}BC^{-1/2})$ denotes the second largest singular value of~$R^{-1/2}BC^{-1/2}$.
   \label{th:nonsquareEML}
\end{theorem} 

Applying Theorem~\ref{th:nonsquareEML} to the adjacency matrix of a graph results in the following, slightly weaker, version of Theorem~\ref{th:strongVolEML}. Since the theorem applies to general nonnegative matrices, an analogous statement can also be derived for directed graphs~\cite{Butler2006}.

\begin{corollary}
    Let~$G=(V,E)$ be a graph. Let~$\sigma_i$ be the~$i^{\text{th}}$ largest normalized adjacency eigenvalue of~$G$ and~$\sigma = \max\{\abs{\sigma_2},\abs{\sigma_n}\}$. Let $\vol{S} = \sum_{V\in S}d_v$ for any set $S\subseteq V$. For every two subsets of vertices~$S,T\subseteq V$, we have 
    \[\abs*{e(S,T) - \frac{\vol{S}\vol{T}}{\vol{V}}}\le \sigma \sqrt{\vol{S}\vol{T}}.\]
    \label{cor:weakVolEML}
\end{corollary}

For the Laplacian matrix, an Expander Mixing Lemma for irregular graphs was proved by Chung (see~\cite[Theorem 7]{chung2004}). This result was already known for disjoint sets~$S,T$ such that~$e(S,T)=0$ due to Haemers~\cite[Lemma 6.1]{H1995}.

\begin{theorem}[Chung {\cite[Theorem 7]{chung2004}}]
Let~$G=(V,E)$ be a graph with Laplacian eigenvalues~$\mu_0\le \dots\le \mu_{n-1}$, and let~$S,T\subseteq V$. Define $\vol{S} = \sum_{V\in S}d_v$ for any $S\subseteq V$. Then 
\begin{align*}&\abs*{e(S,T) + \frac{(\mu_1+\mu_{n-1})}{2} \left(\abs{S\cap T} - \frac{|S||T|}{n} \right) -\vol{S\cap T}} \\
&\ \ \ \ \le \frac{\mu_1-\mu_{n-1}}{2n} \sqrt{|S||T|\left(n-|S|\right)\left(n-|T|\right)}.
\end{align*}
\label{th:LaplacianEML}
\end{theorem}

Theorem~\ref{th:LaplacianEML} was used to find a connection between several graph parameters and Laplacian eigenvalues, such as toughness~\cite{gu2022}, integrity and tenacity~\cite{chen2022}.

\section{New Expander Mixing Lemmas for irregular graphs}\label{sec:expandermixingnew}

In this section, we provide our main result, an Expander Mixing Lemma for nonnegative matrices. This result can be seen as a unifying framework, which encompasses most known results from Section~\ref{sec:existing}. We use it to derive a new version of the Expander Mixing Lemma with Perron weights and generalize the~$k$-Expander Mixing Lemma from~\cite{ACT2016} to irregular graphs. We also extend an Expander Mixing Lemma for 1-designs by Haemers~\cite{haemersThesis,H1995} to hold for general bipartite graphs and explore its implications for non-bipartite graphs through the double cover.

\subsection{A general EML for nonnegative matrices}\label{subsec:matrixEML}

As we have seen in Section~\ref{sec:existing}, many versions of the Expander Mixing Lemma have been proposed in the literature. In this section, inspired by the results of Butler~\cite{Butler2006} and Nikiforov~\cite{nikiforov2009}, we provide a more general theorem to unify these statements. 

\begin{theorem}
    Let~$B\in \mathbb{R}^{m\times n}$ be a nonnegative matrix without zero rows or columns and let~$\vecx\in \mathbb{R}^n$ and~$\vecy\in \mathbb{R}^m$ be positive vectors. Let~$C$ and~$R$ be diagonal matrices such that~$B\vecx=R\vecy$ and~$\vecy^{\top}B = \vecx^{\top} C$. Then for all~$S\subseteq [m]$ and~$T\subseteq [n]$,
    \begin{align*}&\abs*{\langle\vecchi_S,B\vecchi_T\rangle - \frac{\langle \vecchi_S,B\vecx\rangle \langle \vecy,B\vecchi_T\rangle}{\langle \vecy,B\vecx\rangle}}\\
    &\le \sigma_2(R^{-1/2}BC^{-1/2})\sqrt{\left(\|R^{1/2}\vecchi_S\|^2-\frac{\langle \vecchi_S,B\vecx\rangle^2}{\langle \vecy,B\vecx\rangle} \right) \left( \|C^{1/2}\vecchi_T\|^2-\frac{\langle \vecy,B\vecchi_T\rangle^2}{\langle \vecy,B\vecx\rangle}\right)},
    \end{align*}
    where~$\sigma_2(R^{-1/2}BC^{-1/2})$ denotes the second largest singular value of~$R^{-1/2}BC^{-1/2}$.
    \label{th:superEML}
\end{theorem}
\begin{proof}
    Let~$M = R^{-1/2}BC^{-1/2}$ and assume without loss of generality that~$m\le n$. Let~$\vecy_1,\dots, \vecy_m \in \mathbb{R}^m$ and~$\vecx_1, \dots, \vecx_n \in \mathbb{R}^n$ be orthonormal bases of left and right singular vectors of~$M$ respectively. Then we can rewrite
    \[R^{1/2}\vecchi_S = \sum_{i=1}^n\alpha_i \vecy_i, \ \ \ C^{1/2}\vecchi_T = \sum_{i=1}^m\beta_i \vecx_i,\]
    with~$\alpha_i = \langle R^{1/2}\vecchi_S, \vecy_i\rangle$ and~$\beta_i = \langle C^{1/2}\vecchi_T, \vecx_i\rangle$. It follows that
    \begin{align}
        \inn{\vecchi_S,B\vecchi_T} &= \inn{R^{1/2}\vecchi_S,MC^{1/2}\vecchi_T}\nonumber\\
        &= \left(\sum_{i=1}^n \alpha_i \vecy_i\right)^{\top}M \left(\sum_{i=1}^m \beta_i \vecx_i\right)\nonumber \\
        &= \left(\sum_{i=1}^n \alpha_i \vecy_i\right)^{\top} \left(\sum_{i=1}^m \sigma_i \beta_i \vecx_i\right)\label{eq:evDecomp}\\
        &= \sum_{i=1}^m \sigma_i(M) \alpha_i \beta_i\nonumber \\
        &= \sigma_1(M)\inn{R^{1/2}\vecchi_S,\vecy_1}\inn{C^{1/2}\vecchi_T,\vecx_1} + \sum_{i=2}^m \sigma_i(M) \alpha_i \beta_i,\nonumber
    \end{align}
   \noindent where~$\sigma_i(M)$ denotes the $i^{\text{th}}$ largest singular value of $M$.
   
    Let~$\vecx'=C^{1/2} \vecx / \sqrt{\inn{\vecy,B\vecx}}$ and $\vecy'=R^{1/2} \vecy / \sqrt{\inn{\vecy,B\vecx}}$. Note that~$M\vecx'= \vecy'$ and~$\vecy'^{\top}M = \vecx'^{\top}$. Moreover,~$\vecx'$ is an eigenvector of~$M^{\top}M$ associated with eigenvalue one. As~$\vecx'$ is positive by assumption, it must be the Perron eigenvector, hence~$\sigma_1 = 1$,~$\vecx_1=\vecx'$ and~$\vecy_1 = \vecy'$. Therefore,
    \[\alpha_1 = \inn{R^{1/2}\vecchi_S,\vecy_1} = \inn{R^{1/2}\vecchi_S,\vecy'} = \inn*{\vecchi_S,\frac{R \vecy}{\sqrt{\inn{\vecy,B\vecx}}}} 
    = \frac{\inn{\vecchi_s,B\vecx}}{\sqrt{\inn{\vecy,B\vecx}}},\]
    and likewise,~$\beta_1 = \inn{C^{1/2}\vecchi_T,\vecx_1} = \inn{\vecy,B\vecchi_T} / \sqrt{\inn{\vecy,B\vecx}}$. This means that
    \begin{align*}
        \abs*{\langle\vecchi_S,B\vecchi_T\rangle - \frac{\langle \vecchi_S,B\vecx\rangle \langle \vecy,B\vecchi_T\rangle}{\langle \vecy,B\vecx\rangle}} &= \abs*{\inn{\vecchi_S,B\vecchi_T} - \sigma_1(M)\inn{R^{1/2}\vecchi_S,\vecy_1}\inn{C^{1/2}\vecchi_T,\vecx_1}}\\
        & = \abs*{\sum_{i=2}^m \sigma_i(M) \alpha_i \beta_i}\\
        &\le \sigma_2(M) \left(\sum_{i=2}^n\alpha_i^2\right)^{1/2}\left(\sum_{i=2}^m\beta_i^2\right)^{1/2}.
    \end{align*}
    The result then follows from the fact that
    \begin{align*}
   & \sum_{i=2}^n\alpha_i^2 = \|R^{1/2}\vecchi_S\|^2 - \alpha_1^2, \ \ \ \sum_{i=2}^m\beta_i^2 = \|C^{1/2}\vecchi_T\|^2 - \beta_1^2.
   \qedhere \end{align*}
\end{proof}

Theorem~\ref{th:superEML} implies the following corollary. This is a stronger version of Theorem~\ref{th:nonsquareEML} by Butler~\cite{Butler2006}, as the right-hand side of the inequality is always smaller or equal to the right-hand side in Theorem~\ref{th:nonsquareEML}.
 
\begin{corollary}
    Let~$B \in \mathbb{R}^{m\times n}$ be a matrix with nonnegative entries and no zero rows/columns. Also, let~$R\in \mathbb{R}^{m\times m}$ and~$C \in \mathbb{R}^{n\times n}$ be the unique diagonal matrices such that $B\j = R\j$ and~$\j^\top B = \j^\top C$. Then for all~$S \subseteq [m]$ and~$T \subseteq [n]$,
   \begin{align*}&\abs*{\langle \vecchi_S,B\vecchi_T \rangle - \frac{\langle \vecchi_S,B\j \rangle\langle\j,B\vecchi_T \rangle}{\langle \j,B\j \rangle}} \\
   &\le \sigma_2(R^{-1/2}BC^{-1/2})\sqrt{ \left(\langle \vecchi_S,B\j \rangle-\frac{\langle \vecchi_S,B\j\rangle^2}{\langle \j,B\j\rangle} \right) \left( \langle \j,B\vecchi_T \rangle-\frac{\langle \j,B\vecchi_T\rangle^2}{\langle \j,B\j\rangle}\right) },
   \end{align*}
   where~$\sigma_2(R^{-1/2}BC^{-1/2})$ denotes the second largest singular value of~$R^{-1/2}BC^{-1/2}$.
\end{corollary}

As a consequence of Theorem~\ref{th:superEML}, we also obtain known Expander Mixing Lemmas for irregular graphs. For example, if we choose~$x=y=\j$ and~$B=A$, Theorem~\ref{th:superEML} reduces to Theorem~\ref{th:strongVolEML}. An overview of the results that follow from Theorem~\ref{th:superEML} is given in Table~\ref{tab:summary}.

In addition to existing results, Theorem~\ref{th:superEML} implies some interesting new Expander Mixing Lemmas, which we will discuss below.

\subsubsection{EML using the Perron eigenvector}

In the classic Expander Mixing Lemma for irregular graphs by Krivelevich, Sudakov~\cite{KS2006} and Chung~\cite{chung2004}, the vertices are weighted by the degree vector. We will derive a similar result which instead uses the Perron eigenvector. 

\begin{theorem}
	 Let $G=(V,E)$ be a graph with adjacency eigenvalues $\lambda_1 \ge \dots \ge \lambda_n$ and normalized Perron eigenvector~$\vecnu$. Let $\lambda = \max\{\abs{\lambda_2},\abs{\lambda_n}\}$. Then for any subsets~$S,T\subseteq V$,
	 \[\big| e(S,T) - \lambda_1 \langle \vecchi_S,\vecnu \rangle\langle\vecchi_T,\vecnu\rangle \big|  \leq \lambda \sqrt{(\abs{S} - \langle\vecchi_S,\vecnu\rangle^2)(\abs{T} - \langle\vecchi_T,\vecnu\rangle^2)}.\]
  \label{th:PerronEML}
	 \end{theorem}
\begin{proof}
This follows from Theorem~\ref{th:superEML} where~$\vecx=\vecy=\vecnu$,~$B=A$ and~$C=R=\lambda_1 I$.
\end{proof}

\noindent In Section \ref{sec:converse} we will prove the converse of Theorem \ref{th:PerronEML}.

If the given graph has an equitable partition, we can use this extra structure to simplify Theorem~\ref{th:PerronEML} as follows.

\begin{corollary}
 Let $G=(V,E)$ be a graph with adjacency eigenvalues $\lambda_1 \ge \dots \ge \lambda_n$ and normalized Perron eigenvector~$\vecnu$. Let $\lambda = \max\{\abs{\lambda_2},\abs{\lambda_n}\}$ and let~$\Pa=\{C_1,\dots,C_m\}$ be an equitable partition of~$G$ such that $S \subseteq C_i$ and~$T\subseteq C_j$ for some~$i,j\in[m]$. Then

	 \[\big| e(S,T) - \lambda_1 \abs{S}\abs{T}\nu_S\nu_T  \big| \leq \lambda \sqrt{\abs{S}\abs{T}(1 - \abs{S}\nu_S^2)(1 - \abs{T}\nu_T^2)}\]
  for constants~$\nu_S$ and~$\nu_T$.
  \label{cor:PerronEMLEquitable}
\end{corollary}
\begin{proof}
This follows directly from the fact that for an equitable partition, $\vecnu$ is constant on the cells, see~\cite[Section 9.3]{GR2001}. 
\end{proof}

It is known that the orbits of the automorphism group induce an equitable partition~\cite[Section 9.3]{GR2001}. Perhaps Corollary~\ref{cor:PerronEMLEquitable} is therefore most interesting in the case where we have a small equitable partition, or a small number of orbits of the automorphism group, as in this case it is more likely that the sets $S$ and $T$ are contained entirely within a cell of the partition.

\subsubsection{EML using the weight-quotient matrix}
\label{subsec:expandermixingnew2}

If we apply Theorem~\ref{th:superEML} to a matrix~$B$ representing a graph,~$\vecx$ and~$\vecy$ can be viewed as some kind of weight vector assigning a weight to each vertex. It therefore makes sense to draw inspiration from the theory of weight-equitable partitions to derive a new weighted version of the Expander Mixing Lemma. Such a result can be obtained from Theorem~\ref{th:superEML} by choosing~$\vecx=\vecy=\vecnu$ and~$B=A\circ \vecnu\vecnu^{\top}$, where~$\circ$ denotes the entry-wise matrix product. However, we provide an alternative, more insightful proof below based on eigenvalue interlacing.

In his PhD thesis, Haemers proved an Expander Mixing Lemma for $1$-designs~\cite[Theorem 5.1]{H1995}. This result implied the well-known upper bound by Vinh on the number of incidences in a point-line set in~$\mathbb{F}_q^2$~\cite{vinh2011}, and its generalization to combinatorial designs by Lund and Saraf~\cite{LS2016} (see \cite{bishnoiblog} for more details). It also provided an Expander Mixing Lemma for biregular graphs, which can be seen as 1-designs where the cells of the bipartition are associated with points and blocks respectively. In this section, we use Haemers' approach to derive a weighted Expander Mixing Lemma for general graphs. We do so using the weight-quotient matrix of a partition, interlacing techniques and the double cover of a graph. 

Given a graph~$G=(V,E)$ with adjacency matrix $A'$, its \textit{double cover} is the bipartite graph with adjacency matrix

\[A = \begin{pmatrix} 0 & A'\\A' & 0\end{pmatrix},\]
i.e., each vertex~$v\in V$ is replaced by two vertices~$v_1,v_2$ and~$v_1\sim w_2$ if~$v\sim w$~in $G$. We will use this idea to show the following result for general (nonbipartite) graphs.

\begin{theorem}
Let~$G=(V,E)$ be a graph with adjacency matrix~$A$, adjacency eigenvalues $\lambda_1 \ge \dots \ge \lambda_n$ and normalized Perron eigenvector~$\vecnu$. Let $\lambda = \max\{\abs{\lambda_2},\abs{\lambda_n}\}$. Let~$S,T\subseteq V$, and~$\tilde{b}^{*}_{ST}=\displaystyle \sum_{u \in S,v \in T}a_{uv}\nu_{u}\nu_{v}$. Then

\[\abs*{\tilde{b}^{*}_{ST} - \lambda_1 \|\rho(S) \|^2 \|\rho(T)\|^2} \le \lambda\norm{\rho(S)}{}\norm{\rho(T)}{}\norm{\rho(\overline{S})}{}\norm{\rho(\overline{T})}{}.\]
	\label{th:weightedEML}
	\end{theorem}

	\begin{proof}
    Let~$G'$ be the double cover of~$G$ on vertex set~$L'\cup R'$, with~$L',R'$ two copies of~$V$. Define~$\overline{S} = L'\setminus S$,~$\overline{T} = R'\setminus T$ and let~$\tilde{B}^{*}$ be the weight-quotient matrix with regard to the partition~$\Pa = \{S,T,\overline{S},\overline{T}\}$ in~$G'$. Denote the $i^{\text{th}}$ largest eigenvalue of~$G'$ by~$\lambda_i(G')$. It then suffices to prove that 
    \begin{equation}\abs*{\tilde{b}^{*}_{ST}\|\rho(R')\|^2 - \lambda_1(G') \|\rho(S) \|^2 \|\rho(T)\|^2} \le \lambda_2(G') \norm{\rho(S)}{}\norm{\rho(T)}{}\norm{\rho(\overline{S})}{}\norm{\rho(\overline{T})}{},\label{eq:HaemersBipartite}\end{equation}
    from which the result follows. 
 
	The normalized weight-quotient matrix of~$\Pa$ is 
	\[\bar{B}^*=\begin{pmatrix}0&0&\bar{b}^{*}_{ST}&\bar{b}^{*}_{S\overline{T}}\\
	0&0&\bar{b}^{*}_{\overline{S}T}&\bar{b}^{*}_{\overline{S}\overline{T}}\\
	\bar{b}^{*}_{TS}&\bar{b}^{*}_{T\overline{S}}&0&0\\
	\bar{b}^{*}_{\overline{T}S}&\bar{b}^{*}_{\overline{T}\overline{S}}&0&0\\
	\end{pmatrix}.\]
	Note that~$\bar{B}^*$ is symmetric, as its entry for cells~$X,Y$ is defined as \[\bar{b}^*_{XY}=\frac{\sum_{u\in X,v\in Y}\nu_u\nu_v a_{uv}}{\|\rho(X)\|\|\rho(Y)\|},\]
    so its determinant equals $\det{\bar{B}^*} = \left(\bar{b}^{*}_{ST}\bar{b}^{*}_{\overline{S}\overline{T}}-\bar{b}^{*}_{S\overline{T}}\bar{b}^{*}_{\overline{S}T}\right)^2$.
	Let~$\mu_1\ge\dots\ge\mu_4$ be the eigenvalues of~$\bar{B}^*$. Note that~$\mu_1 = -\mu_4$ and~$\mu_2 = -\mu_3$. Applying Lemma~2.3 from~\cite{F1999}, we can use eigenvalue interlacing to obtain~$\lambda_1(G')\ge \mu_1$,~$\lambda_2(G')\ge \mu_2$ and~$\lambda_{n-1}(G') \le \mu_3$. Therefore,
	\begin{align*} 
 \frac{\left(\bar{b}^{*}_{ST}\bar{b}^{*}_{\overline{S}\overline{T}}-\bar{b}^{*}_{S\overline{T}}\bar{b}^{*}_{\overline{S}T}\right)^2}{\lambda_1(G')^2} &= \frac{\det{\bar{B}^*}}{\lambda_1(G')^2} \le \frac{\det{\bar{B}^*}}{\mu_1^2} = \frac{\det{\bar{B}^*}}{-\mu_1\mu_4}\\ &= -\mu_2\mu_3 \le \lambda_2(G')\lambda_{n-1}(G') =\lambda_2(G')^2.
 \end{align*}
	Taking the square root on each side and substituting~$\bar{b}^*_{XY} = \tilde{b}^*_{XY}/(\norm{\rho(X)}{}\norm{\rho(Y)}{})$ gives 
 \[\abs*{\frac{\tilde{b}^{*}_{ST}\tilde{b}^{*}_{\overline{S}\overline{T}}-\tilde{b}^{*}_{\overline{S}T}\tilde{b}^{*}_{S\overline{T}}}{\lambda_1(G')}} \le \lambda_2(G') \norm{\rho(S)}{}\norm{\rho(T)}{}\norm{\rho(\overline{S})}{}\norm{\rho(\overline{T})}{}.\]
 Furthermore, note that
 \[\tilde{b}^{*}_{\overline{S}\overline{T}} = \sum_{u\in\overline{S} } \sum_{v\in\overline{T} } a_{u,v}\nu_u\nu_v = \sum_{v\in\overline{T} }\nu_v\sum_{u\in G(v) }\nu_u  - \sum_{u\in S}\sum_{v\in\overline{T} }a_{u,v}\nu_u \nu_v = \lambda_1(G')\|\rho(\overline{T})\|^2-\tilde{b}^{*}_{S\overline{T}},\]
 and similarly,~$\tilde{b}^{*}_{\overline{S}T} = \lambda_1(G')\|\rho(T)\|^2-\tilde{b}^{*}_{ST}$,~$\tilde{b}^{*}_{S\overline{T}} = \lambda_1(G')\|\rho(S)\|^2-\tilde{b}^{*}_{ST}$. It then follows that
 \begin{align*}
    \tilde{b}^{*}_{ST}\tilde{b}^{*}_{\overline{S}\overline{T}}-\tilde{b}^{*}_{\overline{S}T}\tilde{b}^{*}_{S\overline{T}} 
    &= \tilde{b}^{*}_{ST} \left( \lambda_1(G')\|\rho(\overline{T})\|^2-\tilde{b}^{*}_{S\overline{T}}\right) -\left(\lambda_1(G')\|\rho(T)\|^2-\tilde{b}^{*}_{ST} \right) \tilde{b}^{*}_{S\overline{T}}\\
    &= \lambda_1(G')\tilde{b}^{*}_{ST}\|\rho(\overline{T})\|^2 -\lambda_1(G')\|\rho(T)\|^2\tilde{b}^{*}_{S\overline{T}}\\
    &= \lambda_1(G') \tilde{b}^{*}_{ST}\left(\|\rho(T)\|^2 + \|\rho(\overline{T})\|^2 \right) - \lambda_1(G')^2 \|\rho(S)\|^2\|\rho(T)\|^2 \\
    &= \lambda_1(G') \tilde{b}^{*}_{ST}\|\rho(R')\|^2 - \lambda_1(G')^2 \|\rho(S)\|^2\|\rho(T)\|^2 ,
 \end{align*}
 which gives the desired inequality.
 \end{proof}

 \begin{example}
     Let~$G$ be a bowtie graph (that is, a graph consisting of two triangles joined at a vertex) and let~$v$ be the vertex of degree four. This graph has spectral parameters
     $\lambda_1 = \frac{1}{2}(1+\sqrt{17})$,~$\lambda = \frac{1}{2}(\sqrt{17}-1)$ and~$\vecnu = (\alpha,\alpha,\alpha,\alpha,\beta)$, with~$\alpha = \frac{1}{2}\sqrt{\frac{1}{2}(1+\frac{1}{\sqrt{17}})}$ and~$\beta = \sqrt{\frac{1}{2}(1-\frac{1}{\sqrt{17}})}$. If~$S = V\setminus \{v\},\ T = \{v\}$, then
     \[\|\rho(S)\| = \|\rho(\overline{T})\| = \sqrt{4\alpha^2} = 2\alpha, \ \ \ \ \|\rho(T)\| = \|\rho(\overline{S})\| = \beta,\ \ \ \ \tilde{b}^{*}_{ST} = 4\alpha\beta = \frac{4}{\sqrt{17}}. \]
     Filling this in in Theorem~\ref{th:weightedEML}, we see that the bound is tight, since~$\alpha^2\beta^2 = \frac{1}{17}$ and hence
     \[\abs*{\tilde{b}^{*}_{ST} - \lambda_1 \|\rho(S) \|^2 \|\rho(T)\|^2} = \frac{2}{17}\left(\sqrt{17}-1\right) = \lambda\norm{\rho(S)}{}\norm{\rho(T)}{}\norm{\rho(\overline{S})}{}\norm{\rho(\overline{T})}{}.\]
     
 \end{example}

\begin{remark}
If there exists an equitable partition~$\Pa=\{C_1,\dots,C_m\}$ such that $S \subseteq C_i$ and~$T\subseteq C_j$ for some~$i,j\in[m]$, Theorem~\ref{th:weightedEML} reduces to Corollary~\ref{cor:PerronEMLEquitable} (note that~$\|\rho(\overline{S})\|=\sqrt{1-\|\rho(S)\|^2}$ and~$\|\rho(\overline{T})\|=\sqrt{1-\|\rho(T)\|^2}$). If~$S,T$ are contained in cells of a weight-equitable partition, we cannot make the same simplifications, since~$\vecnu$ is only constant on each cell if the partition is also equitable. However, we do know that~$\tilde{b}^*_{ST} = b^*_{ST}\|\rho(S)\|^2$. 
\end{remark}

For biregular graphs, Theorem~\ref{th:weightedEML} gives the following known result.
    
\begin{corollary}[{\cite[Theorem 3.1.1]{{haemersThesis}}\cite[Theorem 5.1]{H1995}}]
Let~$G=(L\cup R,E)$ be a biregular graph with degrees~$d_L$ and~$d_R$ and adjacency eigenvalues $\lambda_1 \ge \dots \ge \lambda_n$. Let~$S\subseteq L$,~$T\subseteq R$ and define~$\overline{S} = L\setminus S$,~$\overline{T} = R\setminus T$. Then
	\[\abs*{\frac{e(S,T)\abs{L}\abs{R}}{|E|} - \abs{S}\abs{T}} \le \frac{\lambda_2}{\sqrt{d_L d_R}} \sqrt{\abs{S}\abs{T}\abs{\overline{S}}\abs{\overline{T}}}.\]
	\label{cor:unweightedBipartiteEML}
\end{corollary}
\begin{proof}
    The result follows from a direct application of Theorem~\ref{th:weightedEML} to the biregular graph~$G$. In this case, the entries of~$\tilde{B}^{*}$ equal the number of edges between the corresponding sets. Moreover, the Perron eigenvector equals~$1/\sqrt{2|L|}$ on~$L$ and~$1/\sqrt{2|R|}$ on~$R$. This reduces Equation~\eqref{eq:HaemersBipartite} in the proof of Theorem~\ref{th:weightedEML} to
    \[\abs*{\frac{e(S,T)}{2} - \sqrt{d_L d_R}\frac{|S||T|}{2\sqrt{|L||R|}}} \le  \frac{\lambda_2}{2\sqrt{|L||R|}}\sqrt{\abs{S}\abs{T}\abs{\overline{S}}\abs{\overline{T}}}.\]
    This expression can be simplified by multiplying each side by~$\sqrt{|L||R|} / 2\sqrt{d_L d_R}$ and substituting~$|E| = d_L |L| = d_R |R|$.
\end{proof}

\subsubsection{$k$-Expander Mixing Lemma}\label{sec:extendedexpandermixinglemma}
The classic Expander Mixing Lemma bounds the number of edges between two sets of vertices in a graph, i.e., the sum of the entries in~$A$ corresponding to these edges. In this section, we consider an Expander Mixing Lemma for polynomials of the adjacency matrix, inspired by the following theorem of Abiad, Ciaob\u{a} and Tait~\cite{ACT2016}.

For $k$ a natural number and~$x\in\mathbb{R}$, let~$x^{(k)} = x + x^2 + \cdots + x^k$. Given a graph~$G$ with adjacency eigenvalues~$\lambda_1\ge \dots \ge \lambda_n$ and a polynomial~$p\in \mathbb{R}_k[x]$, we define
\[W(p) = \max_{u\in V} \{(p(A))_{uu}\}, \ \ \ \ \lambda(p) = \min_{i\in [2,n]}{p(\lambda_i)}, \ \ \ \ \Lambda(p) = \max_{i\in [2,n]}{p(\lambda_i)}.\]

\begin{theorem}[$k$-Expander Mixing Lemma, {\cite[Theorem 4.3]{ACT2016}}]
Let $G=(V,E)$ be a $d$-regular graph. For $S, T\subseteq V$, let $W_k(S,T)$ be the number of walks of length at most $k$ with one endpoint in $S$ and one endpoint in $T$. Then for any $S,T\subseteq V(G)$, we have
\begin{equation*}
\left|W_k(S,T) - \frac{d^{(k)} |S||T|}{n}\right| \leq \lambda^{(k)} \sqrt{ |S||T| \left(1 - \frac{|S|}{n}\right)\left(1 - \frac{|T|}{n}\right)} < \lambda^{(k)} \sqrt{|S||T|}.
\end{equation*}
\label{th:kEML}
\end{theorem}

Note that Theorem~\ref{th:kEML} can also be derived from Theorem~\ref{th:superEML} by substituting~$B=A^{(k)}$ and~$\vecx=\vecy=\j$. In \cite[Theorem 4.2]{ACT2016}, it was used to obtain a Hoffman-type bound on the~$k$-independence number, see Section~\ref{application:hoffmantypebound}.

The number of walks of length at most~$k$ between vertices~$v,w$ equals the~$v,w$-entry of the matrix~$A+A^2+\dots +A^k$. Theorem~\ref{th:kEML} can be extended to other degree-$k$ polynomials in~$A$, which will allow us to derive better bounds on the $k$-independence number using a more general polynomial. Moreover, it can be generalized to irregular graphs using the Perron eigenvector, similarly as we did for Theorem~\ref{th:PerronEML}. We combine these generalizations in Theorem~\ref{th:generalPolyEML}\ref{th:generalPolyEMLii}. Additionally, we prove a stronger version for the case~$S=T$ in Theorem~\ref{th:generalPolyEML}\ref{th:generalPolyEMLi}, which we will use in Section \ref{application:hoffmantypebound} to derive an upper bound on the~$k$-independence number. 

\begin{theorem}
Let $G=(V,E)$ be a graph with adjacency matrix $A$, adjacency eigenvalues $\lambda_1 \ge \dots \ge \lambda_n$ and normalized Perron eigenvector~$\vecnu$. Let $\lambda = \max\{\abs{\lambda_2},\abs{\lambda_n}\}$,~$p\in\mathbb{R}_k[x]$ and $\lambda(p),\Lambda(p)$ as defined above. For subsets $S, T\subseteq V$, define $P(S,T)\coloneqq\sum_{i\in S,j\in T}p(A)_{ij}$. Then
\begin{enumerate}[label=(\roman*)]
    \item $-\Lambda(p)(|S|-\langle\vecchi_S,\vecnu\rangle^2)\le  p(\lambda_1)\langle\vecchi_S,\vecnu\rangle^2 - P(S,S) \le -\lambda(p) \left(\abs{S} - \langle\vecchi_S,\vecnu\rangle^2\right);$\label{th:generalPolyEMLi}
    \item if~$p(A)$ has nonnegative entries,~$|p(\lambda)|\ge |p(\lambda_i)|$ for all~$i \ge 2$ and~$p(\lambda_1) > |p(\lambda)|$, then for any $S,T\subseteq V(G)$, we have
\begin{align*}
\left|P(S,T) - p(\lambda_1)\langle\vecchi_S,\vecnu\rangle \langle\vecchi_T,\vecnu\rangle\right| &\leq |p(\lambda)| \sqrt{(\abs{S}-\langle\vecchi_S,\vecnu\rangle^2)(|T|-\langle\vecchi_T,\vecnu\rangle^2)}\\&< |p(\lambda)|\sqrt{|S||T|}.
\end{align*}\label{th:generalPolyEMLii}
\end{enumerate}
\label{th:generalPolyEML}
\end{theorem}
\begin{proof}
For Claim \ref{th:generalPolyEMLi}, let $\vecx_1=\vecnu, \vecx_2, \dots, \vecx_n$ be an orthonormal basis of eigenvectors for $A$. This is also an orthonormal basis of eigenvectors for $p(A)$ with corresponding eigenvalues $p(\lambda_1),p(\lambda_2),\dots,p(\lambda_n)$ (note however that these may not be in nonincreasing order like we normally assume). Let~$\alpha_i = \langle \vecchi_S,\vecnu\rangle$, then
\begin{align*}
    P(S,S) &= \left(\sum_{i=1}^n\alpha_i \vecx_i\right)^{\top} p(A) \left(\sum_{i=1}^n\alpha_i \vecx_i\right)\\
    &= \sum_{i=1}^n\alpha_i^2 p(\lambda_i)\\
    &= \langle \vecchi_S,\vecnu\rangle^2 p(\lambda_1) + \sum_{i=2}^n\alpha_i^2 p(\lambda_i), 
\end{align*}
and hence
\[p(\lambda_1)\langle \vecchi_S,\vecnu\rangle^2 - P(S,S) = - \sum_{i=2}^n\alpha_i^2 p(\lambda_i) \le -\lambda(p) \sum_{i=2}^n\alpha_i^2 = -\lambda(p)(|S|-\langle \vecchi_S,\vecnu\rangle^2).\]
The lower bound follows analogously.

Claim \ref{th:generalPolyEMLii} can be obtained directly from Theorem~\ref{th:superEML} by choosing~$\vecx=\vecy=\vecnu$ and~$B=p(A)$.
\end{proof}
By choosing~$p(x) = x^{(k)}$, we get the following generalization of Theorem~\ref{th:kEML} for irregular graphs. 

\begin{corollary}
Let $G=(V,E)$ be a graph on~$n$ vertices with adjacency eigenvalues $\lambda_1 \ge \dots \ge \lambda_n$ and normalized Perron eigenvector~$\vecnu$. Let $\lambda = \max\{\abs{\lambda_2},\abs{\lambda_n}\}$ and $\lambda^{(k)}$ as defined above. For $S, T\subseteq V$, let $W_k(S,T)$ be the number of walks of length at most $k$ with one endpoint in $S$ and one endpoint in $T$. Then for any $S,T\subseteq V$, we have
{\small{
\begin{equation*}
\left|W_k(S,T) - \lambda_1^{(k)}\langle\vecchi_S,\vecnu\rangle \langle\vecchi_T,\vecnu\rangle\right| \leq \lambda^{(k)} \sqrt{(\abs{S}-\langle\vecchi_S,\vecnu\rangle^2)(|T|-\langle\vecchi_T,\vecnu\rangle^2)} < \lambda_k^{(k)} \sqrt{|S||T|}.
\end{equation*}
}}
\label{cor:irregWalkEML}
\end{corollary}

\section{Applications}\label{sec:applications}

In this section, we illustrate some applications of the new Expander Mixing Lemmas 
by deriving new bounds on several NP-hard graph parameters: the zero forcing number, the routing number and the $k$-independence number. Moreover, our lower bounds on the zero forcing number also lower bound the vertex integrity of a graph. For fairness of comparison, we only consider expander mixing lemmas for irregular graphs which are based on adjacency eigenvalues: Theorem~\ref{th:strongVolEML}, Theorem~\ref{th:PerronEML} and Theorem~\ref{th:weightedEML}. We will not use Lemma~\ref{lem:taitEML}, as we want the bounds in this section to also hold for graphs which do not satisfy the condition~$\lambda \ll \overline{d}$. The effectiveness of the zero forcing bounds depends on the principal ratio of the graph, which we briefly discuss at the end of Section~\ref{sec:prinRat}. 

\subsection{Zero forcing number and vertex integrity}
\label{sec:zeroForcing}

Zero forcing is a propagation process on a graph where an initial subset of vertices is
colored black, while the others are colored white. A black vertex~$v$ colors (\emph{forces}) a white neighbor~$w$ black if~$w$ is the only white neighbor of~$v$. The \textit{zero forcing number} of a graph~$G$, denoted by~$Z(G)$, is the cardinality of a smallest subset of initial black vertices which colors the entire graph black.

Many NP-hard graph parameters can be approximated or bounded using the eigenvalues of the adjacency matrix. As zero forcing has a close relation to the spectrum  through maximum nullity (see for example~\cite{AIM} for more details), it makes sense to search for spectral upper and lower bounds of~$Z(G)$. In particular, we are interested in lower bounds, as these are generally much more difficult to find. Such bounds have previously been studied in~\cite{chen2023},~\cite{kalinowski2019zero} and \cite{zhang2022}. In this section, we prove three spectral lower bounds for irregular graphs and show that there exist infinite families of graphs for which they are tight.

In~\cite{kalinowski2019zero} the following spectral lower bound on the zero forcing number of~$d$-regular graphs was shown. 

\begin{theorem}[{\cite[Theorem 1.4 (i)]{kalinowski2019zero}}]
Let~$G$ be a $d$-regular graph on~$n$ vertices with adjacency eigenvalues~$\lambda_1\ge \dots \ge \lambda_n$ and let~$\lambda = \max\{\abs{\lambda_2},\abs{\lambda_n}\}$. Then
\[Z(G) \ge n\left(1-\frac{2\lambda}{d+\lambda}\right).\]\label{th:spectralRegCorrected}
\end{theorem}
Theorem~\ref{th:spectralRegCorrected} only holds for regular graphs. We will combine its proof idea with an application of Theorem \ref{th:strongVolEML}, Theorem~\ref{th:PerronEML} and Theorem~\ref{th:weightedEML} to obtain three new lower bounds on the zero forcing number that extend Theorem~\ref{th:spectralRegCorrected} to general graphs. Before we can state our main results of this section, we need some preliminary definitions.

A vertex~$v$ \textit{dominates}~$w$ if~$w\in G[v]$. A sequence~$(v_1,v_2,\dots,v_\ell)$ of vertices in~$V$ is \emph{dominating} for~$G$ if all~$w\in V$ are dominated by some~$v_i$. If additionally for each~$i$ it holds that 
\[G[v_i]\setminus\bigcup_{j=1}^{i-1}G[v_j]\neq \emptyset,\]
i.e., no vertex in the sequence is redundant,~$(v_1,v_2,\dots,v_\ell)$ is called a \emph{Grundy dominating sequence}. A vertex~$w$ is \emph{footprinted} by~$v_i$ if~$i$ is the smallest integer such that~$v_i$ dominates~$w$. The \emph{Grundy domination number} of~$G$, denoted by~$\gamma_{gr}(G)$, is defined as the maximum length of a Grundy dominating sequence. It was shown in~\cite{BRESAR2017} that for any graph without isolated vertices,~$Z(G) \ge |V| - \gamma_{gr}(G)$.

\begin{theorem}
	Let~$G$ be a graph on~$n$ vertices with minimum and maximum degree~$\delta, \Delta$, adjacency eigenvalues~$\lambda_1\ge \dots \ge \lambda_n$, normalized adjacency eigenvalues $\sigma_1\ge \dots \ge \sigma_n$ and  normalized Perron eigenvector~$\vecnu$. Let~$\nu_{\min} = \min \nu_i,\ \nu_{\max} = \max \nu_i$,~$\lambda = \max\{\abs{\lambda_2},\abs{\lambda_n}\}$ and~$\sigma = \max\{\abs{\sigma_2},\abs{\sigma_n}\}$. Then
	\begin{enumerate}[label=(\roman*)]
		\item $Z(G) \ge n\left(1-\frac{2\sigma\Delta}{\delta\left(\delta/\Delta+\sigma\right)}\right)$;\label{th:ZIrregular1}
        \item $Z(G) \ge n\left(1-\frac{2\lambda}{n(\lambda+\lambda_1)\nu_{\min}^2}\right)$;\label{th:ZIrregular3}
        \item $Z(G) \ge n\left( 1-\frac{2\lambda}{\lambda_1(\nu_{\min}/\nu_{\max})^4 + \lambda}\right)$.\label{th:ZIrregular2}
	\end{enumerate}
	
	\label{th:ZIrregular}
\end{theorem}

\begin{proof}
	Let~$v_1,\dots,v_k$ be a maximum Grundy dominating sequence of~$G$ which footprints vertices~$w_1,\dots,w_k$. We will prove the case of even and odd~$k$ separately.
\begin{description}
\item[Case $k$ even.] Let~$k=2\mu n$. By definition, there are no edges between~$S\coloneqq\{v_1,\dots,v_{\mu n}\}$ and~$T\coloneqq\{w_{\mu n + 1},\dots, w_{2\mu n}\}$. Moreover,~$S$ and~$T$ are disjoint. Theorem~\ref{th:strongVolEML} then implies that
	\begin{align*}
		0 &= e(S,T) \\
    &\ge \frac{\vol{S}\vol{T}}{\vol{V}} -\sigma \sqrt{\vol{S}\vol{T}\left(1-\frac{\vol{S}}{\vol{V}}\right)\left(1-\frac{\vol{T}}{\vol{V}}\right)}\\
		&\ge \mu^2 n \frac{\delta^2}{\Delta} - \sigma\mu n\Delta \left(1-\frac{\delta\mu}{\Delta}\right)\\ 
		&= \mu n\left( \mu\frac{\delta^2}{\Delta} -\sigma\Delta\left(1-\frac{\delta\mu}{\Delta}\right)\right),
	\end{align*}
	from which we obtain
	\[\mu \le \frac{\sigma\Delta}{\delta(\delta / \Delta + \sigma)}.\]
	Therefore, 
	\[Z(G) \ge n-k \ge n\left(1-\frac{2\sigma\Delta}{\delta\left(\delta / \Delta + \sigma\right)}\right),\]
	proving Claim \ref{th:ZIrregular1}. To prove \ref{th:ZIrregular3} and \ref{th:ZIrregular2}, we replace Theorem~\ref{th:strongVolEML}  by Theorem~\ref{th:PerronEML} and Theorem~\ref{th:weightedEML}, respectively. This gives 
 \begin{align*}0 &= e(S,T) \ge \lambda_1\langle \vecchi_S,\vecnu\rangle \langle \vecchi_T,\vecnu\rangle -\lambda\sqrt{(|S|-\langle \vecchi_S,\vecnu\rangle^2)(|T|-\langle \vecchi_T,\vecnu\rangle^2)}\\
 &\ge \lambda_1 \nu_{\min}^2\mu^2n^2-\lambda(\mu n - \mu^2 n^2 \nu_{\min}^2)\\
 &= \mu n\left( \lambda_1 \nu_{\min}^2\mu n - \lambda(1-\mu n \nu_{\min}^2) \right)
    \end{align*}
 and 
 \begin{align*}0 &= \tilde{b}^*_{ST} \ge \lambda_1 \norm{\rho(S)}{}^2\norm{\rho(T)}{}^2-\lambda \norm{\rho(S)}{}\norm{\rho(T)}{}\norm{\rho(\overline{S})}{}\norm{\rho(\overline{T})}{}\\
    &\ge \lambda_1 |S|\nu_{\min}^2 |T|\nu_{\min}^2-\lambda \sqrt{\nu_{\max}^2|S| \nu_{\max}^2|T| \nu_{\max}^2|\overline{S}| \nu_{\max}^2|\overline{T}|}\\
    &= \lambda_1 \mu^2n^2\nu_{\min}^4-\lambda \nu_{\max}^4\sqrt{\mu^2n^2(n-\mu n)^2}\\
    &= \mu n^2\left( \lambda_1 \mu  \nu_{\min}^4 - \lambda \nu_{\max}^4(1-\mu)\right).
    \end{align*}
    The remainder of the proof is analogous to~\ref{th:ZIrregular1}.
\item[Case $k$ odd.] Let~$k=2\mu n+1$. There are no edges between~$S\coloneqq\{v_1,\dots,v_{\mu n}\}$ and~$T\coloneqq\{w_{\mu n + 1},\dots, w_{2\mu n+1}\}$ and~$S$ and~$T$ are disjoint. Theorem~\ref{th:strongVolEML} then implies that
   \begin{align*}
		0 &= e(S,T)\\
  &\ge \frac{\vol{S}\vol{T}}{\vol{V}} -\sigma \sqrt{\vol{S}\vol{T}\left(1-\frac{\vol{S}}{\vol{V}}\right)\left(1-\frac{\vol{T}}{\vol{V}}\right)}\\
		&\ge \frac{\delta^2\mu n(\mu n + 1)}{\Delta n} - \sigma\Delta\sqrt{\mu n (\mu n + 1) \left(1-\frac{\delta\mu n}{\Delta n}\right) \left(1-\frac{\delta(\mu n + 1)}{\Delta n}\right)}\\ 
        &\ge \frac{\delta^2\mu n(\mu n + 1)}{\Delta n} - \sigma\Delta \left(1-\frac{\delta\mu n}{\Delta n}\right)\sqrt{\mu n (\mu n + 1) }\\
        &\ge \sqrt{\mu n (\mu n + 1) } \left( \frac{\delta^2\sqrt{\mu n(\mu n + 1)}}{\Delta n} - \sigma\Delta \left(1-\frac{\delta\mu}{\Delta}\right)\right),
	\end{align*}
    from which we obtain
	\[0 \ge \frac{\delta^2\sqrt{\mu n(\mu n + 1)}}{\Delta n} - \sigma\Delta \left(1-\frac{\delta\mu}{\Delta}\right)\ge \frac{\delta^2\mu n}{\Delta n} - \sigma\Delta \left(1-\frac{\delta\mu}{\Delta}\right).\]
    The remainder of the proof is identical to the first case.
\qedhere\end{description}
\end{proof}

For~$d$-regular graphs, all three bounds in Theorem~\ref{th:ZIrregular} imply Theorem~\ref{th:spectralRegCorrected}, hence they all generalize Theorem~\ref{th:spectralRegCorrected}, the known result for regular graphs from~\cite{kalinowski2019zero}.

Note that if~$\Delta = k\delta$, we need~$1/(2k^2-k) > \sigma$ in order for Theorem~\ref{th:ZIrregular}\ref{th:ZIrregular1} to be positive. To avoid putting strong restrictions on~$\sigma$, it therefore makes sense to look at `almost' regular graphs, i.e., graphs with~$\Delta = \delta+O(1)$. A family of such graphs for which the bound is asymptotically tight are the complete graphs~$K_n$ with one edge removed. These have parameters~$\Delta = \delta+1 = n-1$,~$\sigma=2/(n+1)$ and the zero forcing number is~$n-2$. Theorem~\ref{th:ZIrregular}\ref{th:ZIrregular1} implies
\[Z(G) \ge n\left(1-\frac{4(n-1)}{n(n-2)} \right)\approx n-4\]
for large~$n$, which is asymptotically tight (note that the tight examples for Theorem~\ref{th:spectralRegCorrected} in~\cite{kalinowski2019zero} are also off by a constant 2). The same asymptotic bound is achieved by Theorem~\ref{th:ZIrregular}\ref{th:ZIrregular3} and~\ref{th:ZIrregular2}. 

We can apply a similar analysis to Theorem \ref{th:ZIrregular}\ref{th:ZIrregular3} and \ref{th:ZIrregular2}. For these bounds to be nontrivial, we need
\[\lambda_1\left(\frac{\nu_{\min}}{\nu_{\max}}\right)^4 > \lambda, \ \ \ n(\lambda_1+\lambda)\nu_{\min}^2>2\lambda.\]
If the graph is close to regular, i.e., $\nu_{\min}, \nu_{\max} \approx\frac{1}{\sqrt{n}}$, these conditions reduce to~$\lambda_1 > \lambda$, which is satisfied for all non-bipartite graphs. However, we do not need almost regular graphs for these bounds to perform well, as is shown by a second family of tight examples arising from the well-known cocktail party graphs. The graph~$H_{(n-2)/2}+K_2$, the join of a $K_2$ and a cocktail party graph on~$n-2$ vertices (with $n$ even), has parameters 
\[\lambda_1 = \frac{n-3+\beta}{2}, \ \ \ \lambda = 2,\ \ \ \nu_{\min} = \frac{1}{2}\sqrt{\frac{2}{n-2}\left(1+\frac{n-5}{\beta}\right)},\]
where~$\beta = \sqrt{(n-3)^2+4n}$. Filling this in in the bound of Theorem~\ref{th:ZIrregular}\ref{th:ZIrregular3} gives a somewhat cumbersome expression which converges to~$n-4$. In fact, computing the bound for small~$n$ shows that it is always exactly equal to~$n-4$ when rounded up to an integer. The zero forcing number of $H_{(n-2)/2}+K_2$ is $n-2$, so this is another infinite family for which Theorem~\ref{th:ZIrregular}\ref{th:ZIrregular3} is asymptotically tight (and again, differs by a constant 2). The same holds for Theorem~\ref{th:ZIrregular}\ref{th:ZIrregular1} and~\ref{th:ZIrregular2}, although unlike~\ref{th:ZIrregular3}, these bounds do not give~$n-4$ for very small values of~$n$. Hence Theorem~\ref{th:ZIrregular}\ref{th:ZIrregular3} performs better for this particular graph class.

\begin{remark}
    In \cite{zhang2022}, Zhang et al.\ propose a spectral bound on the zero forcing number that involves only $\lambda_1$. In general, this bound seems to be incomparable to ours. For example, consider the graph~$(K_3\square K_3)+K_2$, where `$\square$' and `$+$' denote the Cartesian product and graph join respectively. This graph's zero forcing number equals five. Theorem~\ref{th:ZIrregular}\ref{th:ZIrregular3} gives a lower bound of four, whereas bound in~\cite{zhang2022} equals two for this graph. On the other hand, there are highly irregular graphs for which~\cite{zhang2022} is better, as all bounds in Theorem~\ref{th:ZIrregular} become negative.
\end{remark}

Let~$G=(V,E)$ be a graph and for any subgraph~$H$, let~$\kappa (H)$ be the largest size of a connected component of~$H$. The \textit{(vertex) integrity} of~$G$ is defined as
\[\iota(G) = \min_{S\subseteq V} (\abs{S} + \kappa(V\setminus S)).\]
It was shown in~\cite{alon2023} that~$\iota(G) \ge n-2x(G)$, with~$x(G)$ the largest integer such that~$G$ contains two disjoint sets of size $x(G)$ with no edges between them. This is essentially the same situation as in the proof of Theorem~\ref{th:ZIrregular}, where~$Z(G) \ge n-2\mu n$ and we apply the Expander Mixing Lemma to sets~$S$ and~$T$ of size~$\mu n$ with~$e(S,T) = 0$. We can therefore adapt the proof of Theorem~\ref{th:ZIrregular} to obtain the following lower bounds on the vertex integrity. However, there does not appear to be a structural connection between this parameter and the zero forcing number.

\begin{theorem}
	Let~$G$ be a graph on~$n$ vertices with minimum and maximum degree~$\delta, \Delta$, adjacency eigenvalues~$\lambda_1\ge \dots \ge \lambda_n$, normalized adjacency eigenvalues $\sigma_1\ge \dots \ge \sigma_n$ and  normalized Perron eigenvector~$\vecnu$. Let~$\nu_{\min} = \min \nu_i,\ \nu_{\max} = \max \nu_i$,~$\lambda = \max\{\abs{\lambda_2},\abs{\lambda_n}\}$ and~$\sigma = \max\{\abs{\sigma_2},\abs{\sigma_n}\}$. Then
	\begin{enumerate}[label=(\roman*)]
		\item $\iota(G) \ge n\left(1-\frac{2\sigma\Delta}{\delta\left(\delta/\Delta+\sigma\right)}\right)$;
        \item $\iota(G) \ge n\left(1-\frac{2\lambda}{n(\lambda+\lambda_1)\nu_{\min}^2}\right)$;
        \item $\iota(G) \ge n\left( 1-\frac{2\lambda}{\lambda_1(\nu_{\min}/\nu_{\max})^4 + \lambda}\right)$.
	\end{enumerate}
	\label{th:integritybounds}
\end{theorem}

Theorem~\ref{th:integritybounds} generalizes the lower bound by Alon et al.\ in~\cite{alon2023} to irregular graphs. The tight examples for Theorem~\ref{th:ZIrregular} are also asymptotically tight for Theorem~\ref{th:integritybounds}, as both graph families have integrity~$n-1$.

\subsubsection*{Minimizing the ratio $\nu_{\max} / \nu_{\min}$} \label{sec:prinRat}

The previous applications of the Expander Mixing Lemma suggest that we should seek for irregular graphs with small ratio $\nu_{\max}/\nu_{\min}$, where $\nu_{\min}$ and~$\nu_{\max}$ are the minimum and maximum entries of the Perron eigenvector $\vecnu$. This ratio is known as the \textit{principal ratio}~$\gamma(G)$. It was shown in~\cite{ostrowski1952}, that
\begin{equation}
    \gamma(G)\ge \sqrt{\frac{\Delta}{\delta}}, \label{eq:principalRatio}
\end{equation} 

\noindent a bound which was later sharpened in~\cite{ciaoba2007}.
This means that if the difference between~$\nu_{\min}$ and~$\nu_{\max}$ is small, the minimum and maximum degree of the graph cannot differ much, so the graph is close to regular. It is also known that the principal ratio satisfies the following upper bound.

\begin{theorem}[{\cite[Theorem 2.1]{ciaoba2007}}]
    Let~$G=(V,E)$ be a graph with normalized Perron eigenvector~$\vecnu$ and let~$u,v$ be the vertices corresponding to the maximum and minimum entries of~$\vecnu$ respectively. If~$u$ and~$v$ are at distance~$d$ in~$G$, then
    \[\gamma(G) \le \frac{\tau^{d+1} - \tau^{-(d+1)}}{\tau - \tau^{-1}},\]
    where~$\tau = \frac{1}{2}(\lambda_1+\sqrt{\lambda_1^2-4})$.
    \label{th:principalRatioUpper}
\end{theorem}

It was shown in~\cite{ciaoba2007} that the lollipop graphs~$P_r \cdot K_s$ (also known as kite graphs) have principal ratio~$\frac{\tau^r-\tau^{-r}}{\tau-\tau^{-1}}$. The authors conjectured that for the right choice of~$r$ and~$s$, these graphs always achieve the largest possible principal ratio among all graphs with the same number of vertices. This was later confirmed in~\cite{tait2015}. 

For our applications of the Expander Mixing Lemma, on the other hand, we are interested in irregular graph families where~$\gamma(G)$ is small. As we saw in Section~\ref{sec:zeroForcing}, this is not required for our bounds to perform well, but it does ensure that the bounds are not trivial. An exhaustive search shows that for small graphs, the family of complete graphs with one edge removed achieves the minimum principal ratio among all irregular graphs. Asymptotically, $\gamma(G) \rightarrow 1$ (this was also observed in~\cite{clark2022}) and these graphs also minimize the ratio~$\Delta / \delta$. It is therefore not surprising that they are tight examples for Theorem~\ref{th:ZIrregular}.

\subsection{Routing number}\label{application:routing}

Let~$G=(V,E)$ be a graph with a pebble placed on each vertex and let~$\pi$ be an arbitrary permutation of $V$. Consider the routing problem where each pebble must be moved from its respective vertex~$v$ to its unique destination~$\pi(v)$ by repeatedly selecting a matching of edges in~$E$ and swapping the pebbles on the endpoints of every matching edge. The \emph{routing number of~$\pi$}, denoted by~$\rt{G,\pi}$, is the minimum number of steps needed to move all pebbles to their destination and the \emph{routing number} of the graph~$G$ is defined as the maximum over all permutations~$\pi$, i.e,
\[\rt{G} = \max_{\pi} \rt{G,\pi}.\]

The routing number was introduced by Alon, Chung and Graham~\cite{acg1993routing}, who proved the exact value for some basic graph classes and obtained a polylogarithmic upper bound for regular expander graphs with a sufficiently large spectral gap. Bounds for trees and hypercubes were found by Zhang~\cite{zhang1999} and Li, Lu and Yang~\cite{Li2010routing} respectively, confirming conjectures by Alon, Chung and Graham~\cite{acg1993routing}. The following spectral upper bound by Horn and Purcilly appeared in~\cite{horn2020routing} and was sharpened by a constant factor in~\cite{purcilly2020thesis}. 

\begin{theorem}[{\cite[Theorem 3]{horn2020routing}, \cite[Theorem 11]{purcilly2020thesis}}]
\label{th:routingThesis}
    Fix $0 < c < 1$. Let $G$ be a graph with minimum degree~$\delta = cn$ and normalized adjacency eigenvalues~$\sigma_1 \ge \dots \ge \sigma_n$. Let~$\sigma = \max\{|\sigma_2|,|\sigma_n|\}$ and suppose~$\sigma < c^2$. Then
    \[\rt{G} \le \frac{12}{c^2(c^2-\sigma)}.\]
\end{theorem}

The proof of this result uses (a weaker version of) the Expander Mixing Lemma for irregular graphs by Chung~\cite{chung2004} (see Theorem~\ref{th:strongVolEML}). Following the same technique, we will derive an alternative upper bound using our Expander Mixing Lemma from Theorem~\ref{th:PerronEML}.

\begin{theorem} \label{th:routingPerron}
    Fix $0 < c < 1$. Let $G$ be a graph with minimum degree~$\delta = cn$, adjacency eigenvalues~$\lambda_1 \ge \dots \ge \lambda_n$ and normalized Perron eigenvector~$\vecnu$. Let~$\lambda = \max\{|\lambda_2|,|\lambda_n|\}$ and suppose~$\lambda < \lambda_1cn\nu_{\min}^2$. Then
    \[\rt{G} \le \frac{12n}{c(\lambda_1cn\nu_{\min}^2-\lambda)}.\]
\end{theorem}
\begin{proof}
Let~$\pi$ be a permutation of~$V(G)$. Then~$\pi=\pi_2 \pi_1$ for some permutations~$\pi_1$,~$\pi_2$ of order two~\cite{acg1993routing}, meaning that~$\pi_1$ and~$\pi_2$ can be written as a product of disjoint transpositions. Let~$(v,v')$ be a transposition in~$\pi_1$ and let~$f(x,y) = \lambda_1 xy\nu_{\min}^2 - \lambda\sqrt{xy}$. For~$x,y \ge cn$, we have
\begin{align*}
    f_x(x,y) &= \lambda_1y\nu_{\min}^2 - \lambda\frac{y}{2\sqrt{xy}}\\
    &= \sqrt{y}\left( \frac{2\lambda_1\nu_{\min}^2\sqrt{xy}-\lambda}{2\sqrt{x}}\right)\\
    &\ge \sqrt{cn}\left(\frac{2\lambda_1\nu_{\min}^2cn - \lambda}{2\sqrt{x}}\right)\\
    &> 0,
\end{align*}
so, by symmetry,~$f$ is increasing in both parameters.

It follows from Theorem~\ref{th:PerronEML} that
\begin{align*}
e(S,T) &\ge \lambda_1 \langle \vecchi_S,\vecnu\rangle  \langle \vecchi_T,\vecnu\rangle -\lambda\sqrt{(|S|- \langle \vecchi_S,\vecnu\rangle^2)(|T|- \langle \vecchi_T,\vecnu\rangle^2)} \nonumber \\ &\ge \lambda_1|S||T|\nu_{\min}^2-\lambda\sqrt{|S||T|}.
\end{align*}
for any~$S,T\subseteq V$. Applying this to~$S=G(v)$ and~$T = G(v')$ gives
\begin{align*}
    e(G(v),G(v')) &\ge \lambda_1|G(v)||G(v')|\nu_{\min}^2-\lambda\sqrt{|G(v)||G(v')|}\\
    &\ge \lambda_1(cn)^2\nu_{\min}^2-\lambda cn\\
    &= cn(\lambda_1 cn \nu_{\min}^2 -\lambda),
\end{align*}
where the second inequality uses the fact that~$f(x,y)$ is increasing for~$x,y\ge cn$ and~$|G(v)|,|G(v')|\ge cn$. 
Let~$\varepsilon = c(\lambda_1 cn \nu_{\min}^2 -\lambda) /(4n)$. Repeating the remainder of the proof of Theorem 11 in~\cite{purcilly2020thesis} with this new value of~$\varepsilon$ gives
\[\rt{G} \le \frac{3}{\varepsilon} = \frac{12n}{c(\lambda_1cn\nu_{\min}^2-\lambda)}.\qedhere \]
\end{proof}

\begin{remark}
With a slightly more careful analysis, the condition~$\sigma<c^2$ in Theorem~\ref{th:routingThesis} can be relaxed to~$\sigma < \frac{(cn)^2}{\vol{V}}$ and the upper bound can be improved to
\[\frac{12}{c^2((cn)^2/\vol{V}-\sigma)}.\]
For a $d$-regular graph, the first inequality simplifies to~$\sigma <d/n$. By substituting~$\nu_{\min}=1/\sqrt{n}$ and~$\lambda = \sigma\lambda_1 = \sigma d$, we find that~$\lambda_1cn\nu_{\min}^2 = d/n$ as well, so for regular graphs, this tighter version of Theorem~\ref{th:routingThesis} and our Theorem~\ref{th:routingPerron} pose the same restrictions on~$G$. Moreover, the upper bounds on the routing number also coincide as 
\[\frac{12}{c^2((cn)^2/\vol{V}-\sigma)} = \frac{12}{c^2(c-\sigma)} = \frac{12n}{c(\lambda_1cn\nu_{\min}^2-\lambda)}.\]
However, if we compare the original statement in Theorem~\ref{th:routingThesis} directly, it is slightly weaker than Theorem~\ref{th:routingPerron} for regular graphs.
\end{remark}

\subsection{$k$-Independence number}\label{application:hoffmantypebound}

The~\textit{$k$-independence number} of a graph~$G$, denoted by~$\alpha_k(G)$, is the maximum size of a set of vertices at pairwise distance greater than $k$. The well-known Hoffman bound on the (1-)in\-de\-pen\-dence number of a graph can be proved using Theorem~\ref{th:irregHaemers} (see Theorem 2.1.2 and 2.1.4 in \cite{haemersThesis}), and in~\cite{ACT2016}, a variant of the Expander Mixing Lemma was used to find a Hoffman-type bound for the~$k$-independence number for any value of~$k$. This bound was improved in~\cite[Theorem 3.2]{ACF2019} using eigenvalue interlacing. In this section, we will show that it can also be obtained as an application of Theorem~\ref{th:generalPolyEML}\ref{th:generalPolyEMLi}.

\begin{corollary}
Let~$G$ be a regular graph and let~$p \in \mathbb{R}_k[x]$ with corresponding parameters~$W(p)$ and~$\lambda(p)$, and assume~$p(\lambda_1) > p(\lambda_i)$ for all $i\le 2$. Then,
\[\alpha_k(G) \le n\frac{W(p)-\lambda(p)}{p(\lambda_1)-\lambda(p)}.\]
\label{cor:boundAlphaK}
\end{corollary}
\begin{proof}
Let~$S$ be a maximum size~$k$-independent set, i.e.,~$\abs{S}=\alpha_k(G)$. By Theorem~\ref{th:generalPolyEML}\ref{th:generalPolyEMLi}, we have
\begin{equation}p(\lambda_1)\frac{\abs{S}^2}{n} - P(S,S) \le -\lambda(p) \left(\abs{S} - \frac{\abs{S}^2}{n}\right).\label{eq:kindEML} \end{equation}
Moreover, since~$S$ is~$k$-independent and~$p$ of degree~$k$,~$p(A)_{ij}=0$ for any distinct~$i,j\in S$, so~$P(S,S)\le \abs{S} W(p)$. Combining this with Equation~\eqref{eq:kindEML} and solving for~$\abs{S}$ gives
\begin{align*}
\abs{S} &= \alpha_k(G) \le n\frac{W(p)-\lambda(p)}{p(\lambda_1)-\lambda(p)}.
\qedhere\end{align*}
\end{proof}
 
\begin{remark} In~\cite[Theorem 4.3]{ACFNZ2022}, it was shown that Corollary~\ref{cor:boundAlphaK} also holds for irregular graphs and the optimal polynomial was computed using mixed integer linear programming. This stronger result does not follow from Theorem~\ref{th:generalPolyEML}\ref{th:generalPolyEMLi}. Following the proof of Corollary~\ref{cor:boundAlphaK}, we find that
\[|S|(\lambda(p)-W(p))\le(\lambda(p)-p(\lambda_1))\langle\vecchi_S,\vecnu\rangle^2.\]
To get rid of~$\vecchi_S$, we can lower bound~$\langle\vecchi_S,\vecnu\rangle^2 \ge |S|^2\nu_{\min}^2$ with ~$\nu_{\min} = \min \vecnu$. Solving for~$|S|=\alpha_k(G)$ gives
\[\alpha_k(G)\le \frac{W(p)-\lambda(p)}{\nu_{\min}^2(p(\lambda_1)-\lambda(p))}. \]
Since~$\nu_{\min}$ is smaller than one for irregular graphs, this is always worse than~\cite[Theorem 4.3]{ACFNZ2022}.
\end{remark}

\section{Converse of the irregular Expander Mixing Lemma}\label{sec:converse}

The Expander Mixing Lemma comes close to characterizing the largest absolute eigenvalue of a graph, as the converse statement is also true: if the number of edges between any two vertex sets is close to the expected distribution in a random graph, the second largest eigenvalue in absolute value is bounded. Variants of this result were described in~\cite[Section 5]{BL2005} and~\cite{bollobas2004}, and were later sharpened in~\cite{lev2015}. 
A converse to the irregular Expander Mixing Lemma in Corollary~\ref{cor:weakVolEML} is provided in~\cite{Butler2006}. In this section, we determine a similar counterpart for Theorem~\ref{th:PerronEML} and Theorem~\ref{th:weightedEML}.

The original proof for regular graphs makes use of the following lemma, which we will also use. 

\begin{lemma}[{\cite[Lemma 3.3]{BL2005}}]
Let~$B$ be an~$n\times n$ real symmetric matrix such that the~$\ell_1$ norm of each row in~$B$ is at most~$d$, and all diagonal entries of~$B$ are, in absolute value, $O(\alpha(\log(d/\alpha)+1))$. Assume that for any two vectors,~$\vecu,\vecv \in \{0,1\}^n$, with~$\supp(\vecu)\cap\supp(\vecv)=\emptyset$:
\[\frac{\abs{\vecu^\top B \vecv}}{\|\vecu\|\|\vecv\|}\le \alpha.\]
Then the spectral radius of~$B$ is~$O(\alpha(\log(d/\alpha)+1))$.
\label{lem:reverseEML}
\end{lemma}

\begin{corollary}[{\cite[Corollary 5.1]{BL2005}} ]
Let~$G$ be a~$d$-regular graph on~$n$ vertices. Suppose that for any~$S,T \subset V(G)$, with~$S\cap T = \emptyset$
\[\abs*{e(S,T) - \frac{\abs{S}\abs{T}d}{n}} \le \alpha \sqrt{\abs{S}\abs{T}}.\]
Then all but the largest eigenvalue of~$G$ are bounded, in absolute value, by $O(\alpha (\log(\frac{d}{\alpha}) + 1))$.
\label{cor:reverseEML}
\end{corollary}

The authors of~\cite{BL2005} also remark that a converse to the bipartite Expander Mixing Lemma for biregular graphs (see Corollary~\ref{cor:unweightedBipartiteEML} where~$d_L=d_R$) can be obtained in a similar manner. This follows directly from the proof of Corollary~\ref{cor:reverseEML} by replacing~$J$ with~$C$, where~$c_{ij} = 0$ whenever~$i,j$ are in the same cell of the bipartition and two otherwise. Using a similar approach, we obtain a converse to Theorem~\ref{th:PerronEML}, the Expander Mixing Lemma based on the Perron eigenvector. Note that eigenvector norms were also used to derive a converse to the Expander Mixing Lemma in~\cite{kenter2016} for~$d$-regular graphs. However, this result is a converse to the classic Expander Mixing Lemma in Theorem~\ref{th:emlOriginal}, whereas our result below is a converse to Theorem~\ref{th:PerronEML}.

\begin{corollary}
Let~$G$ be a graph on~$n$ vertices with normalized Perron eigenvector~$\vecnu$. Suppose that for any~$S,T \subset V(G)$, with~$S\cap T = \emptyset$
\[\left | e(S,T) - \lambda_1 \langle \vecchi_S,\vecnu \rangle\langle\vecchi_T,\vecnu\rangle \right| \leq \alpha \sqrt{\abs{S}\abs{T}}.\]
Then all but the largest eigenvalue of~$G$ are bounded, in absolute value, by~$O(\alpha(\log(\frac{d}{\alpha})+1))$.
\label{cor:reverseIrregEML}
\end{corollary}
\begin{proof}
Let~$B=A-\lambda_1M$, where~$M=\vecnu \vecnu^\top$. If~$\vecnu,\vecx_2,\dots,\vecx_n$ is an orthonormal basis of eigenvectors for~$A$, then
\[(M\vecnu)_j = \nu_j\sum_{i=1}^n\nu_i^2 = \nu_j, \ \ \ (M\vecx_i)_j = \nu_j\langle \vecnu,\vecx_i\rangle = 0.\]
In other words,~$B$ also has eigenvectors~$\vecnu,\vecx_2,\dots,\vecx_n$ with eigenvalues~$0,\lambda_2,\dots,\lambda_n$. The result then follows from Lemma~\ref{lem:reverseEML} with~$\vecu=\vecchi_S$ and~$\vecv=\vecchi_T$.
\end{proof}

Since Lemma~\ref{lem:reverseEML} assumes~$\vecu,\vecv\in \{0,1\}^n$, we cannot use this lemma to obtain a converse to Theorem~\ref{th:weightedEML}, as this requires~$\vecu$ and~$\vecv$ to be defined as
\[u_i = \left\{
    \begin{array}{ll} \nu_i & \text{if }i\in S\\
    0 & \text{otherwise,}
    \end{array}
 \right. \ \ \ \ \ \ v_i = \left\{
    \begin{array}{ll} \nu_i & \text{if }i\in T\\
    0 & \text{otherwise.}
    \end{array}
 \right. \]
Instead, a converse can be obtained directly from~\cite[Theorem 2]{Butler2006} by substituting~$B=A\circ \vecnu \vecnu^\top$ (where~$\circ$ denotes the entry-wise product).

\begin{corollary}
Let~$G=(V,E)$ be a graph on~$n$ vertices with adjacency matrix~$A$ and normalized Perron eigenvector~$\vecnu$. Suppose that for any~$S,T \subset V$
\[\left | \tilde{b}^*_{ST} - \lambda_1 \|\rho(S)\|^2\|\rho(T)\|^2 \right| \leq \alpha \|\rho(S)\|\|\rho(T)\|,\]
where~$\tilde{b}^*_{ST} = \sum_{u\in S, v\in T}a_{uv}\nu_u \nu_v$.
Then~$\lambda/\lambda_1 = O(\alpha(\log(1/\alpha)+1))$.
\label{cor:reverseHaemersEML}
\end{corollary}

\section{Concluding remarks}

The Expander Mixing Lemma is a powerful tool in spectral graph theory and the study of pseudorandom
structures. Many variations of it have been introduced in the literature. In this article, we propose an overarching theorem which encompasses these known versions and implies several interesting new ones. Table~\ref{tab:summary} provides an overview of existing and new Expander Mixing Lemmas for general graphs, and how they follow from our main result (Theorem~\ref{th:superEML}).

{\footnotesize{
\begin{table}[htb!]
\centering
\noindent\begin{tabular}{|l|l c|}
    \hline
     Choice of parameters in Theorem~\ref{th:superEML} & \multicolumn{2}{l|}{Consequence}\\
      \hline
      $G$ graph, $\vecx=\vecy=\vecnu$, $B=A$ & Theorem~\ref{th:PerronEML} &\\
      $G$ graph, $\vecx=\vecy=\j$, $B=A$ & Theorem~\ref{th:strongVolEML} & \cite{chung2004}\\
      $G$ graph, $\vecx=\vecy=\vecnu$, $B=A\circ \vecnu\vecnu^{\top}$ & Theorem~\ref{th:weightedEML} &\\
      $G$ $d$-regular graph, $\vecx=\vecy=\j$, $B=A^{(k)}$ & Theorem~\ref{th:kEML} & \cite{ACT2016}\\
      $G$ graph, $\vecx=\vecy=\vecnu$, $B=P(A)$ & Theorem~\ref{th:generalPolyEML} &\\
      $B>0$, no zero rows or columns, $\vecx=\vecy=\j$ & Stronger version & \cite{Butler2006}\\
       & of Theorem~\ref{th:nonsquareEML} &\\
      \hline
\end{tabular}
\caption{An overview of the variants of the Expander Mixing Lemma for general graphs and their relation to the new Theorem~\ref{th:superEML}. }
\label{tab:summary}
\end{table}
}}

In Section \ref{sec:applications}, known bounds on the zero forcing number, vertex integrity, routing number and $k$-independence number are extended to irregular graphs using the new versions of the Expander Mixing Lemma based on the Perron eigenvector. A natural open problem to explore next, is to find other NP-hard graph problems where this technique leads to new results. In particular, we are interested in applications where the sets~$S$ and~$T$ are not disjoint, since in this case the Laplacian version by Chung (Theorem~\ref{th:LaplacianEML}) has many extra terms in the left-hand side, and hence is not easily applicable. Moreover, it would be interesting to find applications of Theorem~\ref{th:superEML} that make use of its generality, involving non-square matrices that need not correspond to graphs. Finally, an interesting open question is whether one can get rid of the assumption that~$B$ must be nonnegative and $\vecx,\vecy$ strictly positive in Theorem~\ref{th:superEML}. While this seems unavoidable in our current proof, the Expander Mixing Lemma for matrices by Nikiforov~\cite{nikiforov2009} requires no such assumptions, so perhaps this result could be acquired with different proof techniques. If this is the case, we could choose $\vecx=-\vecy=\frac{1}{\mu}\j$, $B=L+\mu I$, $\mu = -\frac{\mu_1+\mu_{n-1}}{2}$, where $\mu_0\le \dots\le \mu_{n-1}$ are the Laplacian eigenvalues of the graph, and so the Expander Mixing Lemma for Laplacian eigenvalues by Chung \cite{chung2004} would also follow as a consequence of our main result.

\subsection*{Acknowledgments}
Aida Abiad is supported by the Dutch Research Council through the grant \linebreak VI.Vidi.213.085.
The authors thank the referee for the many detailed comments.
The authors would also like to thank Sam Adriaensen and Mike T8 for inspiring discussions. 



\begin{thebibliography}{99}

\bibitem{A2019} A. Abiad. A characterization and an application of weight-regular partitions of graphs. \emph{Linear Algebra and its Applications} 569 (2019), 162--174.

\bibitem{ACF2019} A. Abiad, G. Coutinho, M.A. Fiol. On the $k$-independence number of graphs. 
\emph{Discrete Mathematics} 342(10) (2019), 2875--2885.

\bibitem{ACFNZ2022} A. Abiad, G. Coutinho, M.A. Fiol, B.D. Nogueira, S. Zeijlemaker. Optimization of eigenvalue bounds for the independence and chromatic number of graph powers. \emph{Discrete Mathematics} 345(3) (2022), 112706. 


\bibitem{ACT2016} A. Abiad, S.M. Ciaob\u{a}, M. Tait.
Spectral bounds for the $k$-independence number of a graph. \emph{Linear Algebra and its Applications} 510 (2016), 160--170.

\bibitem{AIM}
{AIM Minimum Rank – Special Graphs Work Group (F.~Barioli, W.~Barrett, S.~Butler, S.~M.~Ciaob\u{a}, D.~Cvetkovi\'c, S.~M.~Fallat, C.~D.~Godsil, W.~H.~Haemers, L.~Hogben,
R.~Mikkelson, S.~Narayan, O.~Pryporova, I.~Sciriha, W.~So, D.~Stevanovi\'c, H.~van der Holst, K.~Vander Meulen, A.~Wangsness Wehe)}.
\newblock Zero forcing sets and the minimum rank of graphs.
\newblock {\em Linear Algebra and its Applications} 428 (2008), 1628--1648.

\bibitem{alon2023}
N. Alon, A. Bishnoi, S. Das, A. Neri. Strong blocking sets and minimal codes from expander graphs. 
\emph{Transactions of the American Mathematical Society} 377(8) (2024), 5389-5410.

\bibitem{AC1988} N. Alon, F.R.K. Chung. Explicit construction of linear sized tolerant networks.
 \emph{Discrete Mathematics} 72 (1988), 15--19.

\bibitem{acg1993routing}
N. Alon, F.R.K. Chung, R.L. Graham. Routing permutations on graphs via matchings. \emph{Proceedings of the Twenty-Fifth Annual ACM Symposium on Theory of Computing} (San Diego, 1993), 583--591.



\bibitem{BL2005} Y. Bilu, N. Linial. Lifts, discrepancy and nearly optimal spectral gap. \emph{Combinatorica} 26(5) (2006), 495-–519.

\bibitem{bishnoiblog} A. Bishnoi.  \emph{
Anurag's Math Blog }, \url{https://anuragbishnoi.wordpress.com/2016/09/26/incidence-bounds-and-interlacing-eigenvalues/}. 

\bibitem{bollobas2004} B. Bollob\'as, V. Nikiforov. Hermitian matrices and graphs: singular values and discrepancy. \emph{Discrete Mathematics} 285(1-3) (2004), 17--32.

\bibitem{BRESAR2017}
B.~Bre{\u{s}}ar, C.~Bujt\'as, T.~Gologranc, S.~Klav\u{z}ar, G.~Ko\u{s}mrlj,
  B.~Patk\'os, Z.~Tuza, and M.~Vizer. Grundy dominating sequences and zero forcing sets. {\em Discrete Optimization} 26 (2017), 66--77, 2017.


\bibitem{BCS1978}  F.C. Bussemaker, D.M. Cvetkovi\'c, J.J. Seidel. Graphs related to exceptional root systems, \emph{Combinatorics} (Proceedings of the Fifth Hungarian Colloquium, Keszthely, 1976), 185--191.

\bibitem{Butler2006} S. Butler. Using discrepancy to control singular values for nonnegative matrices. \emph{Linear Algebra and its Applications} 419(2-3) (2006), 486--493.
  
\bibitem{BT2023} J. Byrne, M. Tait. Improved upper bounds on even-cycle creating Hamilton paths. \emph{Discrete Mathematics} 347(10) (2024), 114107.

\bibitem{ciaoba2007}
S.M. Ciaob\u{a}, D.A. Gregory. Principal eigenvectors of irregular graphs. \emph{The Electronic Journal of Linear Algebra} 16 (2007), 366--379.

\bibitem{chen2022}
H. Chen, J. Li. $\ell$-Connectivity, integrity, tenacity, toughness and eigenvalues of graphs. \emph{Bulletin of the Malaysian Mathematical Sciences Society} 45(6) (2022), 3307--3320.

\bibitem{chen2023} H. Chen, J. Li, S.-J. Xu. Spectral bounds for the zero forcing number of a graph. \emph{Discussiones Mathematicae Graph Theory} 44 (2024), 971--982.

\bibitem{chung2004}
F. Chung. Discrete isoperimetric inequalities. \emph{Surveys in Differential Geometry} 9(1) (2004), 53--82.

\bibitem{clark2022}
G.J. Clark. Comparing eigenvector and degree dispersion with the principal ratio of a graph. \emph{Linear and Multilinear Algebra} (2022), 1--15.

 

\bibitem{fan2012spectral} Y.Z. Fan, G.D. Dong. Spectral condition for a graph to be Hamiltonian with respect to normalized Laplacian. arXiv:1207.6824, 2012.

\bibitem{F1999} M.A. Fiol. Eigenvalue interlacing and weight parameters of graphs. \emph{Linear Algebra and its Applications} 290(1) (1999), 275--301.

\bibitem{FG1999} M.A. Fiol, E. Garriga. On the algebraic theory of pseudo-distance-regularity around a set. \emph{Linear Algebra and its Applications} 298(1) (1999), 115--141.



\bibitem{GR2001} C.D. Godsil, G.F. Royle. 
\emph{Algebraic Graph Theory}, volume 207. Springer, 2001.


\bibitem{gu2022} X. Gu, W.H. Haemers.
Graph toughness from Laplacian eigenvalues. \emph{Algebraic Combinatorics} 5(1) (2022), 53--61.
 
\bibitem{haemers1979eigenvalue} W.H. Haemers. Eigenvalue methods. Packing and covering in combinatorics (A. Schrijver ed.), \emph{Mathematical Centre Tract} 106 (1979), 15--38.

\bibitem{haemersThesis} W.H. Haemers. Eigenvalue techniques in design and graph theory. \emph{Mathematical Centre Tract} 121 (1980), Mathematical Centre.

\bibitem{H1995} W.H. Haemers. Interlacing eigenvalues and graphs. \emph{Linear Algebra and its Applications} 226 (1995), 593--616. 

\bibitem{hansen2021} J. Hansen. Expansion in matrix-weighted graphs. \emph{Linear Algebra and its Applications} 630 (2021), 252--273.


\bibitem{horn2020routing} P. Horn, A. Purcilly. Routing number of dense and expanding graphs. \emph{Journal of Combinatorics} 11(2) (2020), 329--350.

\bibitem{huang2022toughness} X. Huang, K.C. Das, S. Zhu. Toughness and normalized Laplacian eigenvalues of graphs. \emph{Applied Mathematics and Computation} 425 (2022), 127075.

\bibitem{kalinowski2019zero} T. Kalinowski, N. Kam\u{c}ev, B. Sudakov. The zero forcing number of graphs. \emph{SIAM Journal on Discrete Mathematics} 33(1) (2019), 95--115. See also: \url{https://people.math.ethz.ch/~sudakovb/erratum-zero-forcing.pdf}.

\bibitem{kenter2016} F.H.J. Kenter. Sharp spectral bounds of several graph parameters using eigenvector norms. \emph{Linear Algebra and its Applications} 504 (2016), 48--63.

\bibitem{KS2006} M. Krivelevich, B. Sudakov. Pseudo-random graphs. \emph{More sets, graphs and numbers.} Bolyai Society Mathematical Studies, 15, Springer, 2006, 199--262.

\bibitem{lev2015} V.F. Lev. Discrete norms of a matrix and the converse to the expander mixing lemma. \emph{Linear Algebra and its Applications} 483 (2015), 158--181.

\bibitem{Li2010routing} W.-T. Li, L. Lu, Y. Yang. Routing numbers of cycles, complete bipartite graphs, and hypercubes. \emph{SIAM Journal on Discrete Mathematics}, 24(4) (2010), 1482--1494.
 


\bibitem{LW2012} G.-S. Lee, C.-W. Weng.
A spectral excess theorem for nonregular graphs. \emph{Journal of Combinatorial Theory, Series A} 119(7) (2012), 1427-1431.

\bibitem{LS2016} 
B. Lund, S. Saraf. Incidence bounds for block designs. \emph{SIAM Journal on Discrete Mathematics} 30(4) (2016), 1997--2010.

\bibitem{nikiforov2009} V. Nikiforov. Cut-norms and spectra of matrices. arXiv:0912.0336, 2009.

\bibitem{ostrowski1952}
A.M. Ostrowski. Bounds for the greatest latent root of a positive matrix. \emph{The Journal of the London Mathematical Society} 27 (1952), 253--256.

\bibitem{purcilly2020thesis}
A. Purcilly. {\em Discrepancy inequalities in graphs and their applications}.
\newblock PhD thesis, University of Denver, 2020.

\bibitem{tait2015}
M. Tait, J. Tobin. Characterizing graphs of maximum principal ratio. \emph{The Electronic Journal of Linear Algebra} 34 (2018), 61--70.

\bibitem{vinh2011}
L.A. Vinh. The Szemerédi–Trotter type theorem and the sum-product estimate in finite fields. \emph{European Journal of Combinatorics} 32(8) (2011), 1177--1181.

\bibitem{zhang1999} L. Zhang. Optimal bounds for matching routing on trees. \emph{SIAM Journal on Discrete Mathematics} 12(1) (1999), 64--77.

\bibitem{zhang2022} 
W. Zhang, J. Wang, W. Wang, S. Ji. On the zero forcing number and spectral radius of graphs. \emph{The Electronic Journal of Combinatorics} 29(1) (2022), \#P1.33.

\end{thebibliography}
\end{document}